\documentclass[11pt,reqno]{amsart}

\usepackage{a4wide}
\usepackage[english]{babel}
\usepackage[utf8]{inputenc}
\usepackage{amsmath,amssymb}
\usepackage[hidelinks]{hyperref}
\usepackage{tikz}
\usepackage{dsfont}
\usepackage{xcolor}

\theoremstyle{definition}
\newtheorem{thm}{Theorem}[section]
\newtheorem{lem}[thm]{Lemma}
\newtheorem{prop}[thm]{Proposition}
\newtheorem{cor}[thm]{Corollary}
\newtheorem{defi}[thm]{Definition}
\newtheorem{rem}[thm]{Remark}
\numberwithin{equation}{section}

\newcommand{\dom}{\operatorname{dom}}
\newcommand{\ran}{\operatorname{ran}}
\newcommand{\sgn}{\operatorname{sgn}}
\newcommand{\Sc}{\operatorname{Sc}}
\renewcommand{\Im}{\operatorname{Im}}
\newcommand{\Arg}{\operatorname{Arg}}
\newcommand{\poly}{{\text{poly}}}
\newcommand{\rest}[1]{\hspace{-0.1cm}\upharpoonright_{#1}}

\title[]{Spectral projectors of bisectorial Clifford operators and applications to the generalized gradient}

\author[Francesco Mantovani]{Francesco Mantovani}
\address{(FM) Politecnico di Milano, Dipartimento di Matematica, Via E. Bonardi 9, 20133 Milano, Italy}
\email{francesco.mantovani@polimi.it}

\author[Peter Schlosser]{Peter Schlosser}
\address{(PS) Graz University of Technology, Institute of Applied Mathematics, Steyrergasse 30, 8010 Graz, Austria}
\email{pschlosser@math.tugraz.at}

\begin{document}

\begin{abstract}
We consider right-linear operators $T$ on a right Hilbert module $V$ over the Clifford algebra $\mathbb{R}_n$, whose $S$-spectrum lies in an acute double sector. For these bisectorial operators, we introduce spectral projectors $P_\pm$ associated with the two cones of the double sector. They decompose the Hilbert module into two submodules $V=V_+\overset{\bullet}{+}V_-$, and the bisectorial operator $T$ into two sectorial operators $T\rest{V_\pm}$. A crucial but non-trivial, cornerstone in this theory is the boundedness of the projectors $P_\pm$, which is, in turn connected to a bounded $H^\infty$-functional calculus of the operator $T$. We provide two practical criteria: either the squared operator admits a bounded $H^\infty$-functional calculus, or the operator is $m$-accretive. Finally, we apply these results to the gradient operator $\nabla_a$ with nonconstant coefficients. For the particular gradient with constant coefficients, we are even able to derive explicit representations of the submodules $V_\pm$ and the projectors $P_\pm$ in Fourier space. Moreover, we identify the sign of the gradient operator with the Clifford-Hilbert transform. This sign plays a central role in the fractional powers of vector operators, which are used, for instance, in the non-local Fourier law of heat propagation.
\end{abstract}

\maketitle

AMS Classification: 47A10, 47A60. \medskip

Keywords: $S$-spectrum, Intrinsic $S$-functional calculus, $H^\infty$-functional calculus, Fractional powers of vector operators, Spectral projectors.

\section{Introduction}

Spectral theory on the $S$-spectrum is by now a well-established framework for quaternionic and Clifford linear operators; see the monographs \cite{FJBOOK,CGK,ColomboSabadiniStruppa2011} and the references therein. Historically, its development was motivated by the quaternionic formulation of quantum mechanics \cite{Adler1995,Birkhoff1936,Emch1963,Finkelstein1962,Horwitz1984}, which requires a spectral theorem for normal operators on quaternionic Hilbert modules \cite{ACK,ColKim,Farenick2003,Sharma1987,Viswanath1971}, and the theory of slice hyperholomorphic functions \cite{ADVCGKS}. One of the most important facts about this theory is that it contains the spectral theory of vector operators, such as the gradient operator \cite{6Global,Gradient,CMS26}, as well as Dirac-type operators \cite{DIRACHYPSPHE,DSS,DiracHarm}. The particular motivation for our general theory, as well as the precise application in Section~\ref{sec_Gradient}, is the gradient with nonconstant coefficients, written as a Clifford scalar operator
\begin{equation*}
\nabla_a=\sum_{i=1}^ne_ia_i(x)\frac{\partial}{\partial x_i}.
\end{equation*}
This gradient turns out to be a bisectorial operator in the Clifford module $L^2(\mathbb{R}^n,\mathbb{R}_n)$, and the holomorphic functional calculus based on the $S$-spectrum allows us to define its fractional powers, see \cite[Section~4]{CMS26}. \medskip

The physical motivation behind fractional powers of the gradient operator is to obtain a non-local Fourier law of heat propagation. Denoting by $v:\mathbb{R}^n\times[0,\infty)\to\mathbb{R}$ the temperature, and by $q:\mathbb{R}^n\times[0,\infty)\rightarrow\mathbb{R}^n$ the heat flow, the classical heat equation $\frac{\partial}{\partial t}v-\Delta v=0$ is deduced from the two laws
\begin{align*}
&q(x,t)=-\nabla v(x,t),\hspace{1.78cm}\text{(Fourier's law)}, \\
&\tfrac{\partial}{\partial t}v(x,t)+\operatorname{div}q(x,t)=0,\qquad\text{(Conservation of energy)}.
\end{align*}
The usual fractional heat equation $\frac{\partial}{\partial t}v+(-\Delta)^\alpha v=0$, $\alpha\in(0,1)$, takes non-local interactions into account by replacing the negative Laplacian with its fractional power. The extensive literature on fractional operators builds upon fundamental works such as \cite{B60,G76,G78a,G78b} and \cite{K60,K66,K67,K69,W61,Y60}. \medskip

Our approach however, is different: we replace the gradient in Fourier's law by its fractional power $p_\alpha(\nabla)$, see \cite{CMS26}, and insert the resulting non-local Fourier law into the unchanged conservation of energy. This keeps the evolution equation in divergence form, so that weak solutions are immediately defined, and it applies to the whole class of gradients $\nabla_a$ with nonconstant coefficients, which model non-homogeneous materials, on bounded as well as on unbounded domains. \medskip

This article is a natural continuation of the recent paper \cite{CMS26} on the fractional powers of bisectorial operators. Here, we investigate the spectral projectors associated with the two cones of the $S$-spectrum of bisectorial Clifford operators. The reason why these particular projectors are important is that in \cite[Theorem~3.28]{CMS26} we found that the connection between the two approaches -- taking fractional powers of the gradient (non-local Fourier's law) and taking fractional powers of the Laplacian (non-local full heat equation) -- is exactly the operator $\sgn(\nabla_a)$. The sign function is constant on each of the two cones of the $S$-spectrum of $\nabla_a$, and so it is naturally connected to the spectral projectors, which are associated with the characteristic functions on the respective cones. \medskip

Since the study of these projectors requires the vector spaces to be one-sided, we employ the intrinsic functional calculus in the spirit of J. Gantner from the AMS memoir \cite{G20}. Using the right multiplication operator $\mathcal{I}^Rs$ from \eqref{Eq_IR}, we found an efficient way to write down the functional calculus of \cite[Definition~3.11]{G20} in the form of our Definition~\ref{defi_omega_Hinfty}~i). Moreover, we extended this $\omega$-functional calculus to an $H^\infty$-functional calculus for polynomially bounded functions in Definition~\ref{defi_omega_Hinfty}~ii). \medskip

The spectral projectors in our theory will be defined via the $H^\infty$-functional calculus \cite{MS24} of the characteristic function of the operator $T$. It is a crucial point of that theory, also in the complex case, that the operator $f(T)$ is not necessarily bounded, even for bounded functions $f$, see \cite{MY90}, \cite[Section~8]{M86}, or \cite[Theorem~10.4.21]{HYTONBOOK2}. However, the projectors need to be bounded, which means that the underlying operator $T$ has to have a bounded $H^\infty$-functional calculus. Although, in \cite{CMS25}, we already gave a characterization of bounded $H^\infty$-functional calculi in terms of quadratic estimates, in this article we will give more accessible sufficient conditions on the operator $T$ for it to have this property. In particular we will show that every $m$-accretive operator has a bounded $H^\infty$-functional calculus, and that $T$ has a bounded $H^\infty$-functional calculus if and only if $T^2$ has one. \medskip

\textit{Main results of this paper.} \medskip

(I) In Section~\ref{sec_Accretive_operators}, we derive a criterion for the boundedness of the $H^\infty$-functional calculus. It is known from \cite[Proposition~3.26]{CMS26} that the square of a bisectorial operator of angle $\omega$ is sectorial of angle $2\omega$. In Lemma~\ref{lem_FC_T2}, we extend this comparison to the level of the functional calculi. More precisely, every function $f$ on the sector $S_{2\theta}$ induces the function $\widehat{f}(s):=f(s^2)$ on the double sector $D_\theta$, and the two functional calculi are connected by
\begin{equation*}
f(T^2)=\widehat{f}(T).
\end{equation*}
As a consequence, Theorem~\ref{thm_Bounded_FC_T2} states that $T$ has a bounded $H^\infty$-functional calculus if and only if $T^2$ has one. This reduces a question about bisectorial operators to a question about sectorial ones, where positivity arguments become available. \medskip

Another class of operators admitting a bounded $H^\infty$-functional calculus are the $m$-accretive operators from Definition~\ref{defi_Accretive}. It is proved in Theorem~\ref{thm_Accretive_boundedFC}, that these operators are always $\frac{\pi}{2}$-sectorial and satisfy the contractivity estimate
\begin{equation*}
\Vert f(T)\Vert\leq\Vert f\Vert_{\infty,S_\theta},\qquad\theta\in(\tfrac{\pi}{2},\pi).
\end{equation*}
In the complex setting, this result is classical. The proof given here is intrinsic to the Clifford module and is based on the $S$-functional calculus of the Clifford adjoint operator \cite{ADJOINT}. Combining the two statements, we obtain a criterion used throughout the rest the paper: if $T$ is injective, bisectorial and $T^2$ is $m$-accretive, then $T$ has a bounded $H^\infty$-functional calculus. \medskip

(II) In Section~\ref{sec_Projectors}, we define the spectral projectors of bisectorial operators. A double sector $D_\theta$ splits into two disjoint sectors $S_\theta$ and -$S_\theta$, and correspondingly, every intrinsic function splits into $f_\pm$, see \eqref{Eq_fpm}. On the operator level, this splitting is realized by the spectral projectors
\begin{equation*}
P_\pm:=\mathds{1}_{\pm S_\theta}(T),
\end{equation*}
defined as the $H^\infty$-functional calculus of the characteristic functions $\mathds{1}_{\pm S_\theta}$ of the two sectors. If the $H^\infty$-functional calculus of $T$ is bounded, which we always assume in this section, these spectral projectors have exactly the properties of their complex counterparts, see Lemma~\ref{lem_P_properties}. Moreover, their ranges $V_\pm:=\ran P_\pm$ induce the direct sum decomposition $V=V_+\overset{\bullet}{+} V_-$ into invariant subspaces, see Lemma~\ref{lem_V_decomposition}. For the restrictions $T_\pm:=T\rest{V_\pm}$, we also prove in Theorem~\ref{thm_Spectrum_restriction} the spectral splitting
\begin{equation*}
\sigma_S(T_\pm)\setminus\{0\}=\sigma_S(T)\cap\mathbb{R}^{n+1}_\pm.
\end{equation*}
In other words, a bisectorial operator with bounded $H^\infty$-functional calculus decomposes into two sectorial operators, one with $S$-spectrum in $\overline{S_\omega}$ and one with $S$-spectrum in $-\overline{S_\omega}$. \medskip

(III) In Section~\ref{sec_Gradient}, we apply our results to the generalized gradient operator $\nabla_a$ in \eqref{Eq_Gradient}. In Theorem~\ref{thm_Bounded_FC_gradient}, we prove that $\nabla_a$ is injective, bisectorial and admits a \textit{bounded} $H^\infty$-functional calculus. In particular, the projectors $P_\pm$ of the gradient are well defined. Moreover, in \eqref{Eq_palpha_gradient} we obtain a factorization of the fractional powers of the gradient into the fractional Laplacian and the bounded sign operator
\begin{equation}\label{Eq_palpha_decomposition}
p_\alpha(\nabla_a)=(\nabla_a^2)^{\frac{\alpha}{2}}\sgn(\nabla_a),\qquad\alpha\in\mathbb{R}.
\end{equation}
This is the precise link between the fractional powers of the vector operator $\nabla_a$ introduced in \cite{CMS26}, and the classical fractional powers of the second-order operator $\nabla_a^2$. \medskip

For the special case of the gradient with constant coefficients, we even calculate in Theorem~\ref{thm_Fourier_P} an explicit representation of the projectors in Fourier space, namely
\begin{equation*}
F[P_\pm u](\xi)=\frac{1}{2}\Big(1\pm\frac{\mathbf{i}\,\xi_a}{|\xi_a|}\Big)F[u](\xi),\qquad u\in L^2(\mathbb{R}^n,\mathbb{R}_n),\,\xi\in\mathbb{R}^n.
\end{equation*}
Also the sign operator in \eqref{Eq_palpha_decomposition} admits the following explicit representation in Fourier space
\begin{equation*}
F[\sgn(\nabla_a)u](\xi)=\frac{\mathbf{i}\,\xi_a}{|\xi_a|}F[u](\xi),\qquad u\in L^2(\mathbb{R}^n,\mathbb{R}_n),\,\xi\in\mathbb{R}^n.
\end{equation*}
This sign operator can also be interpreted as the Clifford Hilbert transform \cite{Bernstein,QianYang2009Hilbert}. Since $\sgn(\nabla_a)$ is bounded and boundedly invertible, the identity \eqref{Eq_palpha_decomposition} also determines the domains of the fractional powers, which, for the gradient, turn out to be the usual Sobolev spaces $H^\alpha(\mathbb{R}^n)$ for $\alpha\geq 0$.

\section{Preliminaries on Clifford Algebras and Clifford modules}

In this section we will fix the algebraic and functional analytic setting of this paper. The underlying algebra will be the real \textit{Clifford algebra} $\mathbb{R}_n$, over $n$ \textit{imaginary units} $e_1,\dots,e_n$, which satisfy the relations
\begin{equation*}
e_i^2=-1\qquad\text{and}\qquad e_ie_j=-e_je_i,\qquad i\neq j\in\{1,\dots,n\}.
\end{equation*}
A special subclass of $\mathbb{R}_n$ will be the so called \textit{paravectors}
\begin{equation*}
\mathbb{R}^{n+1}:=\big\{x_0+x_1e_1+\dots+x_ne_n\;\big|\;x_0,x_1,\dots,x_n\in\mathbb{R}\big\}.
\end{equation*}
Moreover, we will define the sphere of \textit{imaginary units} as
\begin{equation*}
\mathbb{S}:=\big\{x_1e_1+\dots+x_ne_n\;\big|\;x_1^2+\dots+x_n^2=1\big\},
\end{equation*}
and for every imaginary unit $J\in\mathbb{S}$, the corresponding \textit{complex hyperplane}
\begin{equation*}
\mathbb{C}_J:=\big\{x+Jy\;\big|\;x,y\in\mathbb{R}\}.
\end{equation*}

\begin{defi}[Right Hilbert module]
Let $V$ be a real vector space, which possesses an additional \textit{right-multiplication} $\,\cdot:V\times\mathbb{R}_n\rightarrow V$, which extends the scalar multiplication by real numbers to Clifford numbers, and satisfy the following properties:
\begin{enumerate}
\item[$\circ$] $(v+w)s=vs+ws$,
\item[$\circ$] $v(s+p)=vs+vp$,\hspace{1cm} $v,w\in V$, $s,p\in\mathbb{R}_n$.
\item[$\circ$] $v(sp)=(vs)p$,
\end{enumerate}
Moreover, we equip this vector space with a \textit{Clifford valued inner product} $\langle\cdot,\cdot\rangle:V\times V\rightarrow\mathbb{R}_n$, that is a $\mathbb{R}$-bilinear Clifford valued form which is compatible with the right multiplication,
\begin{enumerate}
\item[$\circ$] $\langle vs,w\rangle=\overline{s}\langle v,w\rangle$,
\item[$\circ$] $\langle v,ws\rangle=\langle v,w\rangle s$,\hspace{1cm} $v,w\in V$, $s\in\mathbb{R}_n$.
\item[$\circ$] $\overline{\langle v,w\rangle}=\langle w,v\rangle$,
\end{enumerate}
In this setting, if $V$, equipped with the scalar part $\Sc\langle\cdot,\cdot\rangle:V\times V\rightarrow\mathbb{R}$, is a real Hilbert space, then we call $V$ a \textit{right Hilbert module}.
\end{defi}

\begin{rem}
Since the Hilbert space structure of a Hilbert module is understood with respect to the scalar part $\Sc\langle\cdot,\cdot\rangle$ of the inner product, also the \textit{induced norm} is understood as the one from the underlying real Hilbert space
\begin{equation*}
\Vert v\Vert^2:=\Sc\langle v,v\rangle,\qquad v\in V.
\end{equation*}
It is not difficult to show that this norm is compatible with the Clifford right multiplication as follows:
\begin{enumerate}
\item[$\circ$] $\Vert vs\Vert\leq 2^{\frac{n}{2}}|s|\Vert v\Vert$,\qquad $s\in\mathbb{R}_n$,
\item[$\circ$] $\Vert vs\Vert=|s|\Vert v\Vert$,\hspace{1.1cm} $s\in\mathbb{R}^{n+1}$.
\end{enumerate}
Note also, that a generalized version of Cauchy-Schwarz inequality holds, namely
\begin{equation*}
|\langle u,v\rangle|\leq 2^{\frac{n}{2}}\Vert u\Vert\Vert v\Vert,\qquad u,v\in V.
\end{equation*}
\end{rem}

In right Hilbert modules we can now define right linear operators.

\begin{defi}[Right linear operators]
Let $V$ be a right Hilbert module. A mapping $T:V\supseteq\dom T\rightarrow V$ is called a \textit{right linear operator}, if its domain $\dom T$ is a right linear subspace of $V$, and if
\begin{enumerate}
\item[$\circ$] $T(v+w)=Tv+Tw$,\qquad $v,w\in\dom T$,
\item[$\circ$] $T(vs)=(Tv)s$,\hspace{1.78cm} $v\in\dom T$, $s\in\mathbb{R}_n$.
\end{enumerate}
Moreover, we call $T$ \textit{closed} (resp. \textit{bounded}), if $T$ is closed (resp. bounded) considered as a real-linear operator in $V$. We will denote the set of bounded operators and the set of closed operators as follows:
\begin{align*}
\mathcal{B}(V):=&\big\{T:V\rightarrow V\;\big|\;T\text{ is right linear, everywhere defined, and bounded}\big\}, \\
\mathcal{K}(V):=&\big\{T:V\supseteq\dom T\rightarrow V\;\big|\;T\text{ is right linear, and closed}\big\}.
\end{align*}
\end{defi}

Different from complex Hilbert spaces, in Cliffordian Hilbert modules, the spectrum of an operator $T$ is connected to the bounded invertibility of the operator
\begin{equation*}
Q_s[T]:=T^2-2s_0T+|s|^2,\qquad\text{with }\dom Q_s[T]:=\dom T^2.
\end{equation*}
This suggests the following definition of $S$-spectrum.

\begin{defi}[$S$-Spectrum]
For every $T\in\mathcal{K}(V)$, its \textit{$S$-resolvent set} and \textit{$S$-spectrum} are defined as
\begin{equation*}
\rho_S(T):=\big\{s\in\mathbb{R}^{n+1}\;\big|\;Q_s[T]^{-1}\in\mathcal{B}(V)\big\}\qquad\text{and}\qquad\sigma_S(T):=\mathbb{R}^{n+1}\setminus\rho_S(T).
\end{equation*}
\end{defi}

It is proven in \cite[Theorem~2.4]{ADJOINT} that the $S$-resolvent set can equivalently be written as
\begin{equation*}
\rho_S(T)=\big\{s\in\mathbb{R}^{n+1}\;\big|\;\mathcal{I}^Rs-T\text{ is bijective}\big\},
\end{equation*}
using for $s\in\mathbb{R}^{n+1}$ the \textit{right multiplication operator} $\mathcal{I}^Rs:V\rightarrow V$, defined as
\begin{equation}\label{Eq_IR}
\mathcal{I}^Rsv:=vs,\qquad v\in V.
\end{equation}
In this article we will consider the important special classes of sectorial and bisectorial operators. In order to introduce them, we define for every $\omega\in(0,\pi)$ the \textit{sector}
\begin{equation}\label{Eq_Somega}
S_\omega:=\big\{s\in\mathbb{R}^{n+1}\setminus\{0\}\;\big|\;|\Arg(s)|<\omega\big\},
\end{equation}
and for every $\omega\in(0,\frac{\pi}{2})$ the \textit{double sector}
\begin{equation}\label{Eq_Domega}
D_\omega:=S_\omega\cup(-S_\omega).
\end{equation}

\begin{defi}[Sectorial and bisectorial operators]
An operator $T\in\mathcal{K}(V)$ is called \textit{sectorial} of angle $\omega\in(0,\pi)$ (resp. \textit{bisectorial} of angle $\omega\in(0,\frac{\pi}{2})$), if its $S$-spectrum is contained in
\begin{equation*}
\sigma_S(T)\subseteq\overline{S_\omega}\qquad\Big(\text{resp}.\,\sigma_S(T)\subseteq\overline{D_\omega}\Big),
\end{equation*}
and for every $\varphi\in(\omega,\pi)$ (resp. $\varphi\in(\omega,\frac{\pi}{2})$), there exists $C_\varphi\geq 0$, such that
\begin{equation}\label{Eq_Resolvent_estimate}
\Vert(\mathcal{I}^Rs-T)^{-1}\Vert\leq\frac{C_\varphi}{|s|},\qquad s\in\mathbb{R}^{n+1}\setminus(S_\varphi\cup\{0\})\;\Big(\text{resp}.\;s\in\mathbb{R}^{n+1}\setminus(D_\varphi\cup\{0\})\Big).
\end{equation}
\end{defi}

For these classes of operators, we will now introduce a holomorphic functional calculus. The notion of holomorphicity in the Clifford algebra $\mathbb{R}_n$ will be the one of so called \textit{intrinsic functions}. While a precise definition is given in \cite[Definition~2.1]{MS24}, intuitively, a function $f:U\rightarrow\mathbb{R}^{n+1}$ is intrinsic, if it is defined on an axially symmetric set $U$, it is complex holomorphic on every hyperplane $\mathbb{C}_J$, and it satisfies $\overline{f(s)}=f(\overline{s})$. \medskip

Furthermore, for the functional calculus to be well defined, we need to restrict intrinsic functions to those with a certain decay/growth at 0 and at infinity. I.e., we define the sets of \textit{decaying}, of \textit{bounded} and of \textit{polynomially growing} functions
\begin{subequations}
\begin{align}
\mathcal{N}^0(U):=&\Big\{f:U\rightarrow\mathbb{R}^{n+1}\text{ intrinsic}\;\Big|\;\exists\alpha>0,\,C_\alpha\geq 0: |f(s)|\leq\frac{C_\alpha|s|^\alpha}{1+|s|^{2\alpha}}\Big\}, \label{Eq_N0} \\
\mathcal{N}^\infty(U):=&\big\{f:U\rightarrow\mathbb{R}^{n+1}\text{ intrinsic}\;\big|\;\exists C\geq 0: |f(s)|\leq C\big\}, \label{Eq_Ninfty} \\
\mathcal{N}^\poly(U):=&\Big\{f:U\rightarrow\mathbb{R}^{n+1}\text{ intrinsic}\;\Big|\;\exists k\geq 0,\,C_k\geq 0: |f(s)|\leq C_k\Big(|s|^k+\frac{1}{|s|^k}\Big)\Big\}. \label{Eq_Npoly}
\end{align}
\end{subequations}
For these classes $\mathcal{N}^0$ and $\mathcal{N}^\poly$ we will now define two functional calculi. For functions in $\mathcal{N}^0$, the calculus will lead to bounded operators, and for $\mathcal{N}^\infty$, it will return an unbounded operator in general. Note that the class $\mathcal{N}^\infty$ is in some sense border line, because on the one hand the calculus of $\mathcal{N}^0$ does no longer apply, but (under the conditions of Section~\ref{sec_Accretive_operators}) the functional calculus of $\mathcal{N}^\infty$ still gives a bounded operator. We will use this space only later in Definition~\ref{defi_Bounded_FC}.

\begin{defi}[$\omega$- and $H^\infty$-functional calculus]\label{defi_omega_Hinfty}
Let $T\in\mathcal{K}(V)$ be injective. Depending on whether the operator $T$ is sectorial or bisectorial, we define the \textit{$\omega$-functional calculus} and the \textit{$H^\infty$-functional calculus} as follows: \medskip

\makebox[0.49\textwidth][l]{
\begin{minipage}[t]{0.47\textwidth}
Let $T$ be sectorial of angle $\omega\in(0,\pi)$. \medskip

i)\;For $f\in\mathcal{N}^0(S_\theta)$, $\theta\in(\omega,\pi)$, we define
\begin{equation*}
f(T)v:=\frac{1}{2\pi}\hspace{-0.2cm}\int\limits_{\partial S_\varphi\cap\mathbb{C}_J}(\mathcal{I}^Rs-T)^{-1}vf(s)ds_J,
\end{equation*}
where $\varphi\in(\omega,\theta)$ and $J\in\mathbb{S}$ are arbitrary.

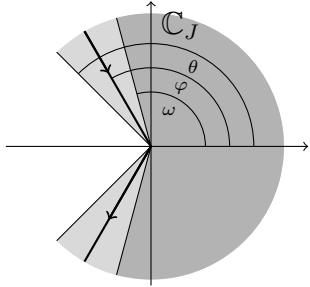
\begin{center}
\begin{tikzpicture}[scale=0.8]
\fill[black!15] (0,0)--(-1.56,-1.56) arc (-135:135:2.2);
\fill[black!30] (0,0)--(-0.57,-2.125) arc (-105:105:2.2);
\draw[thick] (-1.1,1.905)--(0,0)--(-1.1,-1.905);
\draw[thick,->] (-1.1,1.905)--(-0.7,1.212);
\draw[thick,->] (0,0)--(-0.7,-1.212);
\draw (-1.56,-1.56)--(0,0)--(-1.56,1.56);
\draw (-0.57,-2.125)--(0,0)--(-0.57,2.125);
\draw (0.9,0) arc (0:105:0.9) (0.3,0.35) node[anchor=south] {\tiny{$\omega$}};
\draw (1.3,0) arc (0:120:1.3) (0.5,0.7) node[anchor=south] {\tiny{$\varphi$}};
\draw (1.7,0) arc (0:135:1.7) (0.7,1.05) node[anchor=south] {\tiny{$\theta$}};
\draw[->] (-2.4,0)--(2.6,0);
\draw[->] (0,-2.3)--(0,2.4);
\draw (0,2) node[anchor=west] {\large{$\mathbb{C}_J$}};
\end{tikzpicture}
\end{center}
\medskip ii)\;For $f\in\mathcal{N}^\poly(S_\theta)$, $\theta\in(\omega,\pi)$, we define
\begin{equation*}
f(T):=e(T)^{-1}(ef)(T),
\end{equation*}
where $e\in\mathcal{N}^0(S_\theta)$ is such that $ef\in\mathcal{N}^0(S_\theta)$ and $e(T)$ is injective. One possible choice is $e(s)=\frac{s^{k+1}}{(1+s)^{2k+2}}$, with $k\in\mathbb{N}_0$ from \eqref{Eq_Npoly}.
\end{minipage}}
\makebox[0.49\textwidth][r]{
\begin{minipage}[t]{0.47\textwidth}
Let $T$ be bisectorial of angle $\omega\in(0,\frac{\pi}{2})$. \medskip

i')\;For $f\in\mathcal{N}^0(D_\theta)$, $\theta\in(\omega,\frac{\pi}{2})$, we define
\begin{equation*}
f(T)v:=\frac{1}{2\pi}\hspace{-0.2cm}\int\limits_{\partial D_\varphi\cap\mathbb{C}_J}(\mathcal{I}^Rs-T)^{-1}vf(s)ds_J,
\end{equation*}
where $\varphi\in(\omega,\theta)$ and $J\in\mathbb{S}$ are arbitrary.
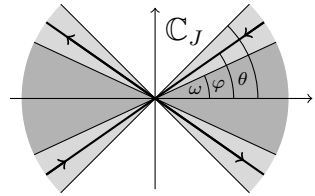
\begin{center}
\begin{tikzpicture}[scale=0.8]
\draw[white,->] (0,-2.3)--(0,2.4);
\fill[black!15] (0,0)--(1.56,1.56) arc (45:-45:2.2)--(0,0)--(-1.56,-1.56) arc (225:135:2.2);
\fill[black!30] (0,0)--(2,0.93) arc (25:-25:2.2)--(0,0)--(-2,-0.93) arc (205:155:2.2);
\draw (2,0.93)--(-2,-0.93);
\draw (2,-0.93)--(-2,0.93);
\draw (1.56,1.56)--(-1.56,-1.56);
\draw (1.56,-1.56)--(-1.56,1.56);
\draw (0.9,0) arc (0:25:0.9) (0.67,-0.1) node[anchor=south] {\tiny{$\omega$}};
\draw (1.3,0) arc (0:35:1.3) (1.05,-0.05) node[anchor=south] {\tiny{$\varphi$}};
\draw (1.7,0) arc (0:45:1.7) (1.45,0.05) node[anchor=south] {\tiny{$\theta$}};
\draw[thick] (1.8,1.26)--(-1.8,-1.26);
\draw[thick] (1.8,-1.26)--(-1.8,1.26);
\draw[thick,->] (1.8,1.26)--(1.47,1.03);
\draw[thick,->] (0,0)--(1.47,-1.03);
\draw[thick,->] (-1.8,-1.26)--(-1.47,-1.03);
\draw[thick,->] (0,0)--(-1.47,1.03);
\draw[->] (-2.4,0)--(2.6,0);
\draw[->] (0,-1.5)--(0,1.5) node[anchor=north west] {\large{$\mathbb{C}_J$}};
\end{tikzpicture}
\end{center}

\medskip ii')\;For $f\in\mathcal{N}^\poly(D_\theta)$, $\theta\in(\omega,\frac{\pi}{2})$, we define
\begin{equation*}
f(T):=e(T)^{-1}(ef)(T),
\end{equation*}
where $e\in\mathcal{N}^0(D_\theta)$ is such that $ef\in\mathcal{N}^0(D_\theta)$ and $e(T)$ is injective. One possible choice is $e(s)=\frac{s^{k+1}}{(1+s^2)^{k+1}}$, with $k\in\mathbb{N}_0$ from \eqref{Eq_Npoly}.
\end{minipage}}
\end{defi}

Although it is not obvious from the definition itself, it turns out the operators $f(T)$ from Definition~\ref{defi_omega_Hinfty} are right-linear with respect to arbitrary Clifford numbers, see Remark~\ref{rem_Gantner_equivalence}. \medskip

To appreciate the differences between the intrinsic $S$-functional calculus and the classic $S$-functional calculus of left slice-hyperholomorphic functions, we provide the following remark, explaining the significance of their distinction.

\begin{rem}\label{rem_Gantner_equivalence}
We point out that the operator $(\mathcal{I}^Rs-T)^{-1}$ is not right linear (only $\mathbb{R}$-linear), and also that the map $s\mapsto(\mathcal{I}^Rs-T)^{-1}$ is not slice hyperholomorphic on $\mathbb{R}^{n+1}$. Still, the definition of the $\omega$-functional calculus makes sense and leads to right linear operators, because it only relies on complex holomorphicity on each complex plane $\mathbb{C}_J$. However, this complex holomorphicity hardly depends on the function $f$ being intrinsic. If one wants to extend the $\omega$-functional calculus also to left slice hyperholomorphic functions, one basically has to replace the resolvent $(\mathcal{I}^Rs-T)^{-1}$ by the left $S$-resolvent operator
\begin{equation}\label{Eq_SL}
S_L^{-1}(s,T)v=Q_s[T]^{-1}\overline{s}v-TQ_s[T]^{-1}v,\qquad v\in V.
\end{equation}
See for example the books \cite[Definition~7.1.5]{FJBOOK} and \cite[Definition~6.2.6]{CGK}, as well as the article \cite[Definition~3.5]{MS24}. Note that the definition \eqref{Eq_SL} itself makes no sense in right Hilbert modules, because the left multiplication $\overline{s}v$ is not defined there. However, it is shown in \cite[Theorem~2.4]{ADJOINT} that the only difference between $(\mathcal{I}^Rs-T)^{-1}$ and $S_L^{-1}(s,T)$ lies precisely in this problematic left product. I.e., if we invert this product, we end up with
\begin{equation*}
(\mathcal{I}^Rs-T)^{-1}=Q_s[T]^{-1}v\overline{s}-TQ_s[T]^{-1}v,\qquad v\in V.
\end{equation*}
\end{rem}

While the $H^\infty$-functional calculus from Definition~\ref{defi_omega_Hinfty}~ii) \& ii') in general defines unbounded operators, it is a central question whether the operator $f(T)$ is bounded for bounded functions $f$. As examples show \cite[Theorem~3]{MY90}, this is not the case in general. However, since this will be a crucial property in our considerations of projectors in this article, we define the following class of operators.

\begin{defi}[Operators with bounded $H^\infty$-functional calclulus]\label{defi_Bounded_FC}
Let $T\in\mathcal{K}(V)$ be injective and sectorial of angle $\omega\in(0,\pi)$ (resp. bisectorial of angle $\omega\in(0,\frac{\pi}{2})$). Then $T$ is said to have a \textit{bounded $H^\infty$-functional calculus}, if for every angle $\theta\in(\omega,\pi)$ (resp. $\theta\in(\omega,\frac{\pi}{2})$), there exists a constant $C_\theta\geq 0$, such that $f(T)\in\mathcal{B}(V)$, and
\begin{equation*}
\Vert f(T)\Vert\leq C_\theta\Vert f\Vert_\infty,\qquad f\in\mathcal{N}^\infty(S_\theta)\;\big(\text{resp. }f\in\mathcal{N}^\infty(D_\theta)\big).
\end{equation*}
\end{defi}

The notion of bounded $H^\infty$-functional calculus is closely related to so called \textit{quadratic estimate}. See \cite[Theorem~4.7]{CMS25}, but also \cite[Section~8]{M86} and \cite[Theorem~10.4.21]{HYTONBOOK2}.

\begin{thm}[Quadratic estimates]\label{thm_Quadratic_estimates}
Let $T\in\mathcal{K}(V)$ be injective and sectorial of angle $\omega\in(0,\pi)$ (resp. bisectorial of angle $\omega\in(0,\frac{\pi}{2})$). Then $T$ has a bounded $H^\infty$-functional calculus if and only if there exists a function $g\in\mathcal{N}^0(S_\theta)$, for some $\theta\in(\omega,\pi)$ (resp. $\theta\in(\omega,\frac{\pi}{2})$), and constants $0<c\leq d$, such that for every $v\in V$ there holds
\begin{equation*}
c\Vert v\Vert^2\leq\int_0^\infty\Vert g(tT)v\Vert^2\frac{dt}{t}\leq d\Vert v\Vert^2\qquad\bigg(\text{resp. }c\Vert v\Vert^2\leq\int_{-\infty}^\infty\Vert g(tT)v\Vert^2\frac{dt}{|t|}\leq d\Vert v\Vert^2\bigg).
\end{equation*}
\end{thm}

\section{Bounded $H^\infty$-functional calculus of $m$-accretive operators}\label{sec_Accretive_operators}

In this section we will specify a class of operators which admit a bounded $H^\infty$-functional calculus, the so called maximal accretive operators from Definition~\ref{defi_Accretive}, also called $m$-accretive in short. We will also study the connection between bisectorial operators $T$ and their sectorial square $T^2$. The main result in this context is Corollary~\ref{cor_T2_accretive}, which is about the bounded $H^\infty$-functional calculus of bisectorial operators with maximal accretive square. This result is then the basis for Theorem~\ref{thm_Bounded_FC_gradient}, which proves that the gradient operator has a bounded $H^\infty$-functional calculus. \medskip

It is already shown in \cite[Lemma~3.25]{CMS26}, that the $S$-spectrum of the squared operator is
\begin{equation}\label{Eq_Spectrum_T2}
\sigma_S(T^2)=\sigma_S(T)^2,
\end{equation}
and in \cite[Proposition~3.26]{CMS26} that the square of every bisectorial operator of angle $\omega\in(0,\frac{\pi}{2})$ is sectorial of angle $2\omega\in(0,\pi)$. In Lemma~\ref{lem_FC_T2}, we follow up with this comparison and also show that the functional calculi of $T$ and $T^2$ are for every function $f:S_{2\theta}\rightarrow\mathbb{R}_n$ connected via the function
\begin{equation*}
\widehat{f}(s):=f(s^2),\qquad s\in D_\theta.
\end{equation*}
In Theorem~\ref{thm_Bounded_FC_T2}, we furthermore show that $T$ has a bounded $H^\infty$-functional calculus, if and only if its square $T^2$ has a bounded $H^\infty$-functional calculus.

\begin{lem}\label{lem_FC_T2}
Let $T\in\mathcal{K}(V)$ be injective and bisectorial of angle $\omega\in(0,\frac{\pi}{2})$. Then for every $f\in\mathcal{N}^\poly(S_{2\theta})$, $\theta\in(\omega,\frac{\pi}{2})$, there is $\widehat{f}\in\mathcal{N}^\poly(D_\theta)$, and there holds
\begin{equation}\label{Eq_FC_T2}
f(T^2)=\widehat{f}(T).
\end{equation}
Here, the left hand side is understood as the sectorial $H^\infty$-functional calculus of the operator $T^2$, and the right hand side as the bisectorial $H^\infty$-functional calculus of the operator $T$.
\end{lem}

\begin{proof}
In the \textit{first step}, let us consider $f\in\mathcal{N}^0(S_{2\theta})$, for some $\theta\in(\omega,\frac{\pi}{2})$. Then it is clear that $\widehat{f}\in\mathcal{N}^0(D_\theta)$, and the sectorial $\omega$-functional calculus of $T^2$ can be written as
\begin{align}
f(T^2)v&=\frac{1}{2\pi}\int_{\partial S_{2\varphi}\cap\mathbb{C}_J}(\mathcal{I}^Rq-T^2)^{-1}vf(q)dq_J \notag \\
&=\frac{1}{2\pi}\int_{\partial S_\varphi\cap\mathbb{C}_J}(\mathcal{I}^Rs^2-T^2)^{-1}v\widehat{f}(s)2sds_J, \label{Eq_FC_T2_1}
\end{align}
where in the second equation we substituted $q=s^2$. By the equivalence of spectra \eqref{Eq_Spectrum_T2}, there is $s\in\rho_S(T)$ if and only if $s^2\in\rho_S(T^2)$, and one can easily verify the resolvent identity
\begin{equation*}
2(\mathcal{I}^Rs^2-T^2)^{-1}\mathcal{I}^Rs=(\mathcal{I}^Rs-T)^{-1}+(\mathcal{I}^Rs+T)^{-1}.
\end{equation*}
With this, we can write \eqref{Eq_FC_T2_1} as
\begin{align*}
f(T^2)v&=\frac{1}{2\pi}\int_{\partial S_\varphi\cap\mathbb{C}_J}\big((\mathcal{I}^Rs-T)^{-1}+(\mathcal{I}^Rs+T)^{-1}\big)v\widehat{f}(s)ds_J \\
&=\frac{1}{2\pi}\int_{\partial S_\varphi\cap\mathbb{C}_J}(\mathcal{I}^Rs-T)^{-1}v\widehat{f}(s)ds_J+\frac{1}{2\pi}\int_{\partial(-S_\varphi)\cap\mathbb{C}_J}(\mathcal{I}^Rs-T)^{-1}v\widehat{f}(-s)ds_J \\
&=\frac{1}{2\pi}\int_{\partial D_\varphi\cap\mathbb{C}_J}(\mathcal{I}^Rs-T)^{-1}v\widehat{f}(s)ds_J=\widehat{f}(T)v,
\end{align*}
where in the last line we used $\widehat{f}(-s)=\widehat{f}(s)$, and that the boundaries of the two sectors combine to the boundary of the double sector $\partial S_\varphi\cup\partial(-S_\varphi)=\partial D_\varphi$. \medskip

In the \textit{second step} we consider a polynomially growing function $f\in\mathcal{N}^\poly(S_{2\theta})$, $\theta\in(\omega,\frac{\pi}{2})$. Then it is clear that $\widehat{f}\in\mathcal{N}^\poly(D_\theta)$ and if $e$ is a regularizer of $f$, then $\widehat{e}$ is a regularizer of $\widehat{f}$. So, the $H^\infty$-functional calculi $\widehat{f}(T)$ and $f(T^2)$ coincide, because
\begin{equation*}
f(T^2)=e(T^2)^{-1}(ef)(T^2)=\widehat{e}(T)^{-1}(\widehat{ef})(T)=\widehat{f}(T). \qedhere
\end{equation*}
\end{proof}

Now that we know that the functional calculi of $T$ and $T^2$ are connected, it is straight forward to conclude that $T$ has a bounded $H^\infty$-functional calculus if and only if $T^2$ has a bounded $H^\infty$-functional calculus, according to Definition~\ref{defi_Bounded_FC}.

\begin{thm}\label{thm_Bounded_FC_T2}
Let $T\in\mathcal{K}(V)$ be injective and bisectorial of angle $\omega\in(0,\frac{\pi}{2})$. Then $T$ has a bounded $H^\infty$-functional calculus if and only if $T^2$ has a bounded $H^\infty$-functional calculus.
\end{thm}

\begin{proof}
Due to the equivalence of bounded $H^\infty$-functional calculus and quadratic estimates in Theorem~\ref{thm_Quadratic_estimates}, it is sufficient to prove for every $f\in\mathcal{N}^0(D_\theta)$, $\theta\in(0,\frac{\pi}{2})$, the integral equality
\begin{equation}\label{Eq_Bounded_FC_1}
\int_0^\infty\Vert f(tT^2)v\Vert^2\frac{dt}{t}=\int_{-\infty}^\infty\Vert\widehat{f}(tT)v\Vert^2\frac{dt}{|t|},\qquad v\in V.
\end{equation}
By \eqref{Eq_FC_T2}, there is $f(tT^2)=\widehat{f}(\sqrt{t}\,T)$, for every $t>0$. With the substitution $\tau=\sqrt{t}$, we can then transform the integral into
\begin{equation*}
\int_0^\infty\Vert f(tT^2)v\Vert^2\frac{dt}{t}=\int_0^\infty\Vert\widehat{f}(\sqrt{t}\,T)v\Vert^2\frac{dt}{t}=2\int_0^\infty\Vert\widehat{f}(\tau T)v\Vert^2\frac{d\tau}{\tau}.
\end{equation*}
Since there is also $\widehat{f}(-\sqrt{t}\,T)=f(tT^2)$ by \eqref{Eq_FC_T2}, we can substitute $\tau=-\sqrt{t}$ in \eqref{Eq_Bounded_FC_1}, and get
\begin{equation*}
\int_0^\infty\Vert f(tT^2)v\Vert^2\frac{dt}{t}=\int_0^\infty\Vert\widehat{f}(-\sqrt{t}\,T)v\Vert^2\frac{dt}{t}=2\int_{-\infty}^0\Vert\widehat{f}(\tau T)v\Vert^2\frac{d\tau}{|\tau|}.
\end{equation*}
Adding these two equations then gives the needed identity \eqref{Eq_Bounded_FC_1}.
\end{proof}

Let us now focus on operators which have a numerical range contained in the right half space. In order to have a large enough domain, we have to add a range condition. In this way we will end up with the following definition of $m$-accretive operators. In Theorem~\ref{thm_Accretive_boundedFC} we will show that $m$-accretive operators are $\frac{\pi}{2}$-sectorial and have a bounded $H^\infty$-functional calculus.

\begin{defi}[$m$-accretive operators]\label{defi_Accretive}
An operator $T\in\mathcal{K}(V)$ is called \textit{$m$-accretive} if it satisfies the following two conditions \medskip
\begin{enumerate}
\item[a)] $\Sc\langle Tv,v\rangle\geq 0$,\qquad $v\in\dom T$, \medskip
\item[b)] $\overline{\ran}(1+T)=V$.
\end{enumerate}
\end{defi}

Let us start with the even more restricted class of bounded, strictly coercive operators.

\begin{lem}\label{lem_Coercive}
Let $T\in\mathcal{B}(V)$ and assume that there exists some $c>0$ such that
\begin{equation}\label{Eq_Coercive}
\Sc\langle Tv,v\rangle\geq c\Vert v\Vert^2,\qquad v\in V.
\end{equation}
Then $T$ is bijective and sectorial of some angle $\omega\in(0,\frac{\pi}{2})$. Moreover, for every $f\in\mathcal{N}^\infty(S_\theta)$, $\theta\in(\frac{\pi}{2},\pi)$, there is $f(T)\in\mathcal{B}(V)$ and admits the representation
\begin{equation}\label{Eq_FC_coercive}
f(T)v=\frac{1}{2\pi}\int_{\partial S_{\frac{\pi}{2}}\cap\mathbb{C}_J}\big((\mathcal{I}^Rs-T)^{-1}-(\mathcal{I}^Rs+T^*)^{-1}\big)vf(s)ds_J,\qquad v\in V.
\end{equation}
\end{lem}

\begin{proof}
In the \textit{first part} we will show that $T$ is bijective and sectorial. To do so, let us consider for every fixed $s\in\mathbb{R}^{n+1}$ with $s_0<c$, the form
\begin{equation*}
q_s(u,v):=\langle(T-\mathcal{I}^Rs)u,v\rangle,\qquad u,v\in V.
\end{equation*}
It is straightforward to verify that $q_s$ is $\mathbb{R}$-linear in the first, and right linear in the second argument. Moreover, the form is bounded, since
\begin{equation*}
|q_s(u,v)|\leq 2^{\frac{n}{2}}\Vert(T-\mathcal{I}^Rs)u\Vert\Vert v\Vert\leq 2^{\frac{n}{2}}(\Vert T\Vert+|s|)\Vert u\Vert\Vert v\Vert,\qquad u,v\in V.
\end{equation*}
Finally, thee scalar part of $q_s$ is elliptic, because for every $v\in V$ we have
\begin{equation*}
\Sc q_s(v,v)=\Sc\langle Tv,v\rangle-\Sc\langle vs,v\rangle=\Sc\langle Tv,v\rangle-s_0\Vert v\Vert^2\geq(c-s_0)\Vert v\Vert^2.
\end{equation*}
In the second equality we used
\begin{equation}\label{Eq_Coercive_2}
\Sc\langle vs,v\rangle=s_0\Vert v\Vert^2,
\end{equation}
which is a consequence of the property $\Sc(\overline{p}\,\overline{q})=\Sc(pq)$ of Clifford numbers, and the resulting
\begin{equation*}
\Sc(s\langle v,v\rangle)=\Sc(\overline{s}\,\overline{\langle v,v\rangle})=\Sc(\overline{s}\langle v,v\rangle).
\end{equation*}
Now we can apply the Lax-Milgram lemma from \cite[Lemma~2.2]{Gradient}. Note that $q_s$ is not right antilinear in the first argument, but it turns out $\mathbb{R}$-linearity is sufficient in the proof of \cite[Lemma~2.2]{Gradient}. Hence, we deduce that $\mathcal{I}^Rs-T$ is bijective for every $s\in\mathbb{R}^{n+1}$ with $s_0<c$, and the inverse operator is bounded by
\begin{equation}\label{Eq_Coercive_1}
\Vert(\mathcal{I}^Rs-T)^{-1}\Vert\leq\frac{1}{c-s_0},\qquad s\in\mathbb{R}^{n+1},\,s_0<c.
\end{equation}
This shows that $s\in\rho_S(T)$ and we have proven that \medskip

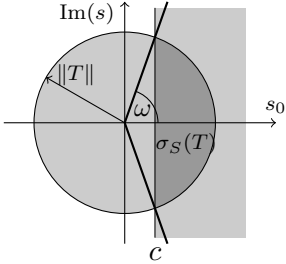
\begin{minipage}{0.35\textwidth}
\begin{center}
\begin{tikzpicture}[scale=0.8]
\fill[black!20] (0,0) circle (1.5cm);
\fill[black!20] (0.5,1.9)--(2,1.9)--(2,-1.9)--(0.5,-1.9);
\fill[black!40] (0.5,1.41) arc (70.53:-70.53:1.5);
\draw (0,0) circle (1.5cm);
\draw[->] (0,0)--(-1.3,0.75);
\draw(-1.3,0.8) node[anchor=west] {\scriptsize{$\Vert T\Vert$}};
\draw (0.5,1.9)--(0.5,-1.9) node[anchor=north] {$c$};
\draw (1.02,0) node[anchor=north] {\scriptsize{$\sigma_S(T)$}};
\draw[thick] (0.707,-2)--(0,0)--(0.707,2);
\draw (0.3,-0.05) node[anchor=south] {\small{$\omega$}};
\draw (0.55,0) arc (0:70.52:0.55);
\draw[->] (-2,0)--(2.5,0) node[anchor=south] {\scriptsize{$s_0$}};
\draw[->] (0,-2)--(0,2);
\draw (0,1.8) node[anchor=east] {\scriptsize{$\Im(s)$}};
\end{tikzpicture}
\end{center}
\end{minipage}
\begin{minipage}{0.64\textwidth}
\begin{equation}\label{Eq_Coercive_12}
\sigma_S(T)\subseteq\big\{s\in\mathbb{R}^{n+1}\;\big|\;s_0\geq c\big\}.
\end{equation}
Since $T$ is a bounded operator, we also know that $\sigma_S(T)$ is contained in the ball of radius $\Vert T\Vert$. Together with \eqref{Eq_Coercive_12} this means that the spectrum is included in the sector
\begin{equation*}
\sigma_S(T)\subseteq\overline{S_\omega},
\end{equation*}
for the choice $\omega=\arccos(\frac{c}{\Vert T\Vert})\in(0,\frac{\pi}{2})$.
\end{minipage}

\medskip In order to verify the resolvent estimate \eqref{Eq_Resolvent_estimate}, we get $\Vert(\mathcal{I}^Rs-T)v\Vert\geq(|s|-\Vert T\Vert)\Vert v\Vert$ from the triangle inequality, which then leads to
\begin{equation*}
\Vert(\mathcal{I}^Rs-T)^{-1}\Vert\leq\frac{1}{|s|-\Vert T\Vert},\qquad s\in\mathbb{R}^{n+1},\,|s|>\Vert T\Vert.
\end{equation*}
Combining this inequality with the resolvent estimate \eqref{Eq_Coercive_1}, gives for every $\varphi\in(\omega,\pi)$ some constant $C_\varphi\geq 0$, with
\begin{equation}\label{Eq_Coercive_3}
\Vert(\mathcal{I}^Rs-T)^{-1}\Vert\leq\frac{C_\varphi}{1+|s|},\qquad s\in\mathbb{R}^{n+1}\setminus S_\varphi.
\end{equation}
This is an even better estimate than the required \eqref{Eq_Resolvent_estimate} for the sectoriality. \medskip

In the \textit{second part} of the proof we will verify the integral representation \eqref{Eq_FC_coercive}. Let us start with $f\in\mathcal{N}^0(S_\theta)$, for which the left hand side of \eqref{Eq_FC_coercive} is the $\omega$-functional calculus \medskip

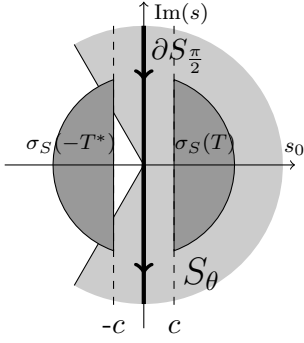
\begin{minipage}{0.35\textwidth}
\begin{center}
\begin{tikzpicture}[scale=0.8]
\fill[black!20] (0,0)--(-1.15,1.992) arc (120:-120:2.3);
\draw (-1.15,-1.992)--(0,0)--(-1.15,1.992);
\draw (0.95,-1.8) node[anchor=center] {\Large{$S_\theta$}};
\draw[fill=black!40] (0.5,1.41) arc (70.53:-70.53:1.5)--(0.5,1.41);
\draw (1.02,0) node[anchor=south] {\scriptsize{$\sigma_S(T)$}};
\draw[fill=black!40] (-0.5,1.41) arc (109.47:250.53:1.5)--(-0.5,1.41);
\draw (-1.2,0) node[anchor=south] {\scriptsize{$\sigma_S(-T^*)$}};
\draw[->] (-2.3,0)--(2.5,0) node[anchor=south] {\scriptsize{$s_0$}};
\draw[->] (0,-2.7)--(0,2.7);
\draw (0,2.5) node[anchor=west] {\scriptsize{$\Im(s)$}};
\draw[dashed] (0.5,2.3)--(0.5,-2.4) node[anchor=north] {$c$};
\draw[dashed] (-0.5,2.3)--(-0.5,-2.4) node[anchor=north] {-$c$};
\draw[ultra thick] (0,-2.3)--(0,2.3);
\draw[ultra thick,->] (0,2.3)--(0,1.4);
\draw[ultra thick,->] (0,0)--(0,-1.8);
\draw (-0.05,1.8) node[anchor=west] {$\partial S_{\frac{\pi}{2}}$};
\end{tikzpicture}
\end{center}
\end{minipage}
\begin{minipage}{0.64\textwidth}
\begin{equation}\label{Eq_Coercive_4}
f(T)v=\frac{1}{2\pi}\int_{\partial S_{\frac{\pi}{2}}\cap\mathbb{C}_J}(\mathcal{I}^Rs-T)^{-1}vf(s)ds_J.
\end{equation}
Next, since we know by \cite[Theorem~3.5]{ADJOINT}, that
\begin{equation*}
\sigma_S(-T^*)=-\sigma_S(T)\subseteq\big\{s\in\mathbb{R}^{n+1}\;\big|\;s_0\leq-c,\,|s|\leq\Vert T\Vert\big\},
\end{equation*}
there is no spectrum of $-T^*$ in the right half plane, where $f$ is holomorphic. Hence there vanishes the integral
\begin{equation}\label{Eq_Coercive_5}
\frac{1}{2\pi}\int_{\partial S_{\frac{\pi}{2}}\cap\mathbb{C}_J}(\mathcal{I}^Rs+T^*)^{-1}vf(s)ds_J=0.
\end{equation}
\end{minipage}

\medskip Substracting now \eqref{Eq_Coercive_5} from \eqref{Eq_Coercive_4}, gives the stated integral \eqref{Eq_FC_coercive} for functions $f\in\mathcal{N}^0(S_\theta)$. \medskip

Parametrising the boundary $\partial S_{\frac{\pi}{2}}\cap\mathbb{C}_J$ by $\gamma(t)=-tJ$, $t\in\mathbb{R}$, this integral can be written as
\begin{equation}\label{Eq_Coercive_6}
f(T)v=\frac{1}{2\pi}\int_\mathbb{R}\underbrace{\big((t\mathcal{I}^RJ+T^*)^{-1}-(t\mathcal{I}^RJ-T)^{-1}\big)}_{=:A(t)}vf(tJ)dt.
\end{equation}
In order to extend this integral identity to all $f\in\mathcal{N}^\infty(S_\theta)$, we need a closer analysis of the operator $A(t)$. Since there is $(\mathcal{I}^RJ)^*=\mathcal{I}^R\overline{J}=-\mathcal{I}^RJ$, we know that $A(t)=A(t)^*$ is self-adjoint, for every $t\in\mathbb{R}$. Moreover, we can rewrite $A(t)$ as
\begin{equation}\label{Eq_Coercive_7}
A(t)=-(t\mathcal{I}^RJ-T)^{-1}(T+T^*)(t\mathcal{I}^RJ+T^*)^{-1}.
\end{equation}
This allows us to show the non-negativity
\begin{align}
\Sc\langle A(t)v,v\rangle&=\Sc\big\langle(T+T^*)(t\mathcal{I}^RJ+T^*)^{-1}v,(t\mathcal{I}^RJ+T^*)^{-1}v\big\rangle \notag \\
&=2\Sc\big\langle T(T^*+t\mathcal{I}^RJ)^{-1}v,(T^*+t\mathcal{I}^RJ)^{-1}v\big\rangle \notag \\
&\geq 2c\Vert(T^*+t\mathcal{I}^RJ)^{-1}v\Vert^2\geq 0, \label{Eq_Coercive_8}
\end{align}
where in the last line we used the coercivity \eqref{Eq_Coercive}. The representation \eqref{Eq_Coercive_7}, in combination with the resolvent estimate \eqref{Eq_Coercive_3}, also allows us to estimate $A(t)$ by
\begin{equation}\label{Eq_Coercive_9}
\Vert A(t)\Vert\leq\frac{2C_{\frac{\pi}{2}}^2\Vert T\Vert^2}{(1+|t|)^2},\qquad t\in\mathbb{R}.
\end{equation}
Let now $f\in\mathcal{N}^\infty(S_\theta)$, $\theta\in(\frac{\pi}{2},\pi)$. First note, since $T$ is bounded and bijective, $f(T)$ is a bounded operator because the $H^\infty$-functional calculus coincides with the bounded $S$-functional calculus from \cite[Section 7.4]{FJBOOK}. Then, we consider for every $n\in\mathbb{N}$ the regularizer
\begin{equation}\label{Eq_Coercive_10}
r_n(s):=\frac{ns}{(s+n)(s+\frac{1}{n})},\qquad s\in S_\theta.
\end{equation}
Then there is $r_nf\in\mathcal{N}^0(S_\theta)$, and by \eqref{Eq_Coercive_6} there holds
\begin{equation*}
(r_nf)(T)v=\frac{1}{2\pi}\int_\mathbb{R}A(t)vr_n(tJ)f(Jt)dt,\qquad n\in\mathbb{N}.
\end{equation*}
In \eqref{Eq_Coercive_9} we have already calculated an integrable upper bound of $\Vert A(t)\Vert$, and the functions $r_n$ are uniformly bounded by
\begin{equation*}
|r_n(tJ)|=\frac{n|t|}{|tJ+n||tJ+\frac{1}{n}|}\leq 1,\qquad t\in\mathbb{R},\,n\in\mathbb{N}.
\end{equation*}
Hence, we can use the dominated convergence theorem to carry the limit inside the integral
\begin{equation*}
\lim\limits_{n\rightarrow\infty}(r_nf)(T)v=\frac{1}{2\pi}\lim\limits_{n\rightarrow\infty}\int_\mathbb{R}A(t)vr_n(tJ)f(tJ)dt=\frac{1}{2\pi}\int_\mathbb{R}A(t)vf(tJ)dt.
\end{equation*}
Moreover, there also converges
\begin{equation*}
r_n(T)v=\underbrace{n(T+n)^{-1}}_{\overset{n\rightarrow\infty}{\longrightarrow}1}\underbrace{T\Big(T+\frac{1}{n}\Big)^{-1}}_{\overset{n\rightarrow\infty}{\longrightarrow}1} v \overset{n\rightarrow\infty}{\longrightarrow}v.
\end{equation*}
Combining these two limits with the product rule $(r_nf)(T)=f(T)r_n(T)$ of the $H^\infty$-functional calculus \cite[Theorem~5.7]{MS24}, it follows from the boundedness of the operator $f(T)$ that for every $v\in V$, there is
\begin{equation*}
f(T)v=\frac{1}{2\pi}\int_\mathbb{R}A(t)vf(tJ)dt. \qedhere
\end{equation*}
\end{proof}

\begin{lem}
Let $T\in\mathcal{K}(V)$ be $m$-accretive. Then $T$ is sectorial of angle $\frac{\pi}{2}$.
\end{lem}

\begin{proof}
For an easier notation let us use the abbreviation
\begin{equation}\label{Eq_Accretive_5}
\mathbb{R}^{n+1}_\pm:=\big\{s\in\mathbb{R}^{n+1}\;\big|\;\pm s_0>0\big\}.
\end{equation}
First, for every $s\in\mathbb{R}^{n+1}_-$, it follows from Definition~\ref{defi_Accretive}~a) and from \eqref{Eq_Coercive_2}, that
\begin{equation*}
\Sc\langle(\mathcal{I}^Rs-T)v,v\rangle=\Sc\langle vs,v\rangle-\Sc\langle Tv,v\rangle\leq s_0\Vert v\Vert^2,\qquad v\in\dom T.
\end{equation*}
Since $s_0<0$, this inequality shows that $\mathcal{I}^Rs-T$ is injective and the inverse is bounded by
\begin{equation}\label{Eq_Accretive_3}
\Vert(\mathcal{I}^Rs-T)^{-1}w\Vert\leq\frac{\Vert w\Vert}{-s_0},\qquad w\in\ran(\mathcal{I}^Rs-T).
\end{equation}
However, we do not yet know if $\mathcal{I}^Rs-T$ is surjective. At least, from $(\mathcal{I}^Rs-T)^{-1}$ being closed and bounded, we know that
\begin{equation}\label{Eq_Accretive_4}
\ran(\mathcal{I}^Rs-T)\text{ is closed}.
\end{equation}
Let us now consider the set
\begin{equation}\label{Eq_Accretive_2}
M:=\big\{s\in\mathbb{R}^{n+1}_-\;\big|\;\ran(\mathcal{I}^Rs-T)=V\big\},
\end{equation}
and show that $M$ as well as its complement $\mathbb{R}^{n+1}_-\setminus M$ are open. \medskip

$\circ$\;\;Let $s\in M$ and consider some arbitrary $p\in\mathbb{R}^{n+1}_-$ with $|p-s|\leq\frac{1}{2\Vert(\mathcal{I}^Rs-T)^{-1}\Vert}$. In order to show that $p\in M$, let us take $w\in\ran(\mathcal{I}^Rp-T)^\perp$, where the orthogonal complement is understood with respect to the real inner product $\Sc\langle\cdot,\cdot\rangle$. Since $\ran(\mathcal{I}^Rs-T)=V$, we can write $w=(\mathcal{I}^Rs-T)v$, for some $v\in\dom T$. From the orthogonality $\Sc\langle w,(\mathcal{I}^Rp-T)v\rangle=0$, there follows from the Pythagorean theorem the identity
\begin{equation}\label{Eq_Accretive_1}
\Vert w\Vert^2=|p-s|^2\Vert v\Vert^2-\Vert(\mathcal{I}^Rp-T)v\Vert^2.
\end{equation}
The second term on the right hand side now admits the lower bound
\begin{equation*}
\Vert(\mathcal{I}^Rp-T)v\Vert\geq\Vert(\mathcal{I}^Rs-T)v\Vert-|p-s|\Vert v\Vert\geq\frac{\Vert v\Vert}{\Vert(\mathcal{I}^Rs-T)^{-1}\Vert}-|p-s|\Vert v\Vert.
\end{equation*}
Plugging this into \eqref{Eq_Accretive_1} gives, after some simplification,
\begin{equation*}
\Vert w\Vert^2\leq\bigg(2|p-s|-\frac{1}{\Vert(\mathcal{I}^Rs-T)^{-1}\Vert}\bigg)\frac{\Vert v\Vert^2}{\Vert(\mathcal{I}^Rs-T)^{-1}\Vert}.
\end{equation*}
Since we chose $|p-s|\leq\frac{1}{2\Vert(\mathcal{I}^Rs-T)^{-1}\Vert}$, the right hand side of this inequality is non-positive. Consequently $w=0$, and we have proven that $\ran(\mathcal{I}^Rp-T)^\perp=\{0\}$, i.e., $\ran(\mathcal{I}^Rp-T)=V$, for every $p$ in a neighbourhood of $s$. This shows that $M$ is open. \medskip

$\circ$\;\;Let $s\in\mathbb{R}^{n+1}_-\setminus M$ and consider some arbitrary $p\in\mathbb{R}^{n+1}_-$ with $|p-s|\leq\frac{1}{\Vert(\mathcal{I}^Rs-T)^{-1}\Vert}$. In order to show that $p\in\mathbb{R}^{n+1}\setminus M$, let us take $w\in\ran(\mathcal{I}^Rp-T)\cap\ran(\mathcal{I}^Rs-T)^\perp$. In particular, we can write $w=(\mathcal{I}^Rp-T)v$, for some $v\in\dom T$. Furthermore, from the orthogonality $\Sc\langle w,(\mathcal{I}^Rs-T)v\rangle=0$ we get a similar identity as in \eqref{Eq_Accretive_1}, and there follows the estimate
\begin{equation*}
\Vert w\Vert^2=|p-s|^2\Vert v\Vert^2-\Vert(\mathcal{I}^Rs-T)v\Vert^2\leq\bigg(|p-s|^2-\frac{1}{\Vert(\mathcal{I}^Rs-T)^{-1}\Vert^2}\bigg)\Vert v\Vert^2.
\end{equation*}
Since we chose $|p-s|\leq\frac{1}{\Vert(\mathcal{I}^Rs-T)^{-1}\Vert}$, the right hand side of this inequality is non-positive. Consequently $w=0$, and we have proven that $\ran(\mathcal{I}^Rp-T)\cap\ran(\mathcal{I}^Rs-T)^\perp=\{0\}$, for every $p$ in a neighbourhood of $s$. However, since $\ran(\mathcal{I}^Rs-T)\neq V$, this is only possible if also $\ran(\mathcal{I}^Rp-T)\neq V$. This shows that $\mathbb{R}^{n+1}_-\setminus M$ is open. \medskip

Altogether, the set $M$ in \eqref{Eq_Accretive_2} is open and has an open complement $\mathbb{R}^{n+1}_-\setminus M$. Since $\mathbb{R}^{n+1}_-$ is connected, this is only possible of either $M=\emptyset$ or $\mathbb{R}^{n+1}_-\setminus M=\emptyset$. However, by \eqref{Eq_Accretive_4} and Definition~\ref{defi_Accretive}~b), there is $-1\in M$, and the only remaining possibility is $\mathbb{R}^{n+1}_-\setminus M=\emptyset$. This proves that $\ran(\mathcal{I}^Rs-T)=V$ for every $s\in\mathbb{R}^{n+1}_-$, and consequently $\mathbb{R}^{n+1}_-\subseteq\rho_S(T)$. \medskip

The resolvent estimate \eqref{Eq_Resolvent_estimate} follows immediately  from \eqref{Eq_Accretive_3}. Indeed, for every $\varphi\in(\frac{\pi}{2},\pi)$ and every $s\in\mathbb{R}^{n+1}\setminus(S_\varphi\cup\{0\})$, there is
\begin{equation*}
\Vert(\mathcal{I}^Rs-T)^{-1}\Vert\leq\frac{1}{-s_0}=\frac{1}{-|s|\cos(\Arg(s))}\leq\frac{1}{-|s|\cos(\varphi)}. \qedhere
\end{equation*}
\end{proof}

\begin{lem}\label{lem_Teps}
Let $T\in\mathcal{K}(V)$ be sectorial of angle $\omega\in(0,\pi)$. Then, for every $\varepsilon\in(0,1]$ the bounded operator
\begin{equation}\label{Eq_Teps}
T_\varepsilon:=(\varepsilon+T)(1+\varepsilon T)^{-1}
\end{equation}
is sectorial of the same angle $\omega$. Moreover, for every $\varphi\in(\omega,\pi)$, there exists a constant $C_\varphi\geq 0$, independent of $\varepsilon\in(0,1]$, such that
\begin{equation}\label{Eq_Teps_resolvent}
\Vert(\mathcal{I}^Rs-T_\varepsilon)^{-1}\Vert\leq\frac{C_\varphi}{|s|},\qquad s\in\mathbb{R}^{n+1}\setminus(S_\varphi\cup\{0\}).
\end{equation}
\end{lem}

\begin{proof}
For every $s\in\mathbb{R}^{n+1}\setminus\overline{S_\omega}$, we can write
\begin{equation}\label{Eq_Teps_1}
\mathcal{I}^Rs-T_\varepsilon=(1-\varepsilon\mathcal{I}^Rs)\Big(\mathcal{I}^R\frac{s-\varepsilon}{1-\varepsilon s}-T\Big)(1+\varepsilon T)^{-1}.
\end{equation}
We can moreover rewrite
\begin{equation}\label{Eq_Teps_2}
\frac{s-\varepsilon}{1-\varepsilon s}=\frac{(1-\varepsilon^2)s}{|1-\varepsilon s|^2}-\varepsilon\Big(1+\frac{(1-\varepsilon^2)|s|^2}{|1-\varepsilon s|^2}\Big).
\end{equation}
This representation shows that $\frac{s-\varepsilon}{1-\varepsilon s}$ is of the form $as-b$, with $a\geq 0$, $b>0$. Since $s\in\mathbb{R}^{n+1}\setminus\overline{S_\omega}$, this representation shows that also $\frac{s-\varepsilon}{1-\varepsilon s}\in\mathbb{R}^{n+1}\setminus\overline{S_\omega}$. Since $\mathbb{R}^{n+1}\setminus\overline{S_\omega}\subseteq\rho_S(T)$, the operator $\mathcal{I}^Rs-T_\varepsilon$ in \eqref{Eq_Teps_1} is bijective for every $s\in\mathbb{R}^{n+1}\setminus\overline{S_\omega}$, with inverse given by
\begin{align*}
(\mathcal{I}^Rs-T_\varepsilon)^{-1}&=(1+\varepsilon T)\Big(\mathcal{I}^R\frac{s-\varepsilon}{1-\varepsilon s}-T\Big)^{-1}(1-\varepsilon\mathcal{I}^Rs)^{-1} \\
&=\mathcal{I}^R\frac{1-\varepsilon^2}{(1-\varepsilon s)^2}\Big(\mathcal{I}^R\frac{s-\varepsilon}{1-\varepsilon s}-T\Big)^{-1}-\mathcal{I}^R\frac{\varepsilon}{1-\varepsilon s},\qquad s\in\mathbb{R}^{n+1}\setminus\overline{S_\omega}.
\end{align*}
Let us now fix $\varphi\in(\omega,\pi)$ and verify the uniform bound \eqref{Eq_Teps_resolvent}. For every $s\in\mathbb{R}^{n+1}\setminus(S_\varphi\cup\{0\})$ there is also $\frac{s-\varepsilon}{1-\varepsilon s}\in\mathbb{R}^{n+1}\setminus(S_\varphi\cup\{0\})$ by \eqref{Eq_Teps_2}, and we can use the sectorial estimate \eqref{Eq_Resolvent_estimate} of the operator $T$, to get
\begin{equation*}
\Vert(\mathcal{I}^Rs-T_\varepsilon)^{-1}\Vert\leq\frac{1-\varepsilon^2}{|1-\varepsilon s|^2}\frac{C_\varphi}{|\frac{s-\varepsilon}{1-\varepsilon s}|}+\frac{\varepsilon}{|1-\varepsilon s|}\leq\frac{c_\varphi^2C_\varphi}{|s|}+\frac{c_\varphi}{|s|},\qquad s\in\mathbb{R}^{n+1}\setminus(S_\varphi\cup\{0\}),
\end{equation*}
using the constant
\begin{equation}\label{Eq_Teps_3}
c_\varphi:=\sup\limits_{p\in\mathbb{R}^{n+1}\setminus(S_\varphi\cup\{0\})}\frac{1}{|1-p|}=\begin{cases} \frac{1}{\sin(\varphi)}, & 0<\varphi\leq\frac{\pi}{2}, \\ 1, & \frac{\pi}{2}\leq\varphi<\pi. \end{cases} \qedhere
\end{equation}
\end{proof}

With these three lemmas we are now in the position to prove the main result of this section, the fact that every $m$-accretive operator admits a bounded $H^\infty$-functional calculus.

\begin{thm}\label{thm_Accretive_boundedFC}
Let $T\in\mathcal{K}(V)$ be injective and $m$-accretive. Then $T$ is sectorial of angle $\frac{\pi}{2}$, and for every $f\in\mathcal{N}^\infty(S_\theta)$, $\theta\in(\frac{\pi}{2},\pi)$, there is $f(T)\in\mathcal{B}(V)$, with
\begin{equation}\label{Eq_Accretive_boundedFC}
\Vert f(T)\Vert\leq\Vert f\Vert_{\infty,S_{\frac{\pi}{2}}}.
\end{equation}
With $\Vert\cdot\Vert_{\infty,S_{\frac{\pi}{2}}}$ we denote the supremum norm over the sector $S_{\frac{\pi}{2}}$.
\end{thm}

\begin{proof}
In the \textit{first step} we will prove the estimate \eqref{Eq_Accretive_boundedFC} for the bounded operator $T_\varepsilon$ in \eqref{Eq_Teps}. With $w:=(1+\varepsilon T)^{-1}v$, we can use the property a) of Definition~\ref{defi_Accretive}, to show its coercivity
\begin{align*}
\Sc\langle(T_\varepsilon-\varepsilon)v,v\rangle&=\Sc\big\langle(\varepsilon+T)w-\varepsilon(1+\varepsilon T)w,v\big\rangle=(1-\varepsilon^2)\Sc\langle Tw,v\big\rangle \\
&=(1-\varepsilon^2)\big(\underbrace{\Sc\langle Tw,w\rangle}_{\geq 0}+\underbrace{\Sc\langle Tw,\varepsilon Tw\rangle}_{=\varepsilon\Vert Tw\Vert^2\geq 0}\big)\geq 0,\qquad v\in V.
\end{align*}
This shows that $T_\varepsilon$ satisfies the assumptions of Lemma~\ref{lem_Coercive}, and hence there is $f(T_\varepsilon)\in\mathcal{B}(V)$, and we can write
\begin{equation}\label{Eq_Accretive_boundedFC_1}
f(T_\varepsilon)v=\frac{1}{2\pi}\int_\mathbb{R}\underbrace{\big((t\mathcal{I}^RJ+T_\varepsilon^*)^{-1}-(t\mathcal{I}^RJ-T_\varepsilon)^{-1}\big)}_{=:A_\varepsilon(t)}vf(tJ)dt,
\end{equation}
see also the representation \eqref{Eq_Coercive_6}. As it is already shown in the inequality \eqref{Eq_Coercive_8} and the text above, the operator $A_\varepsilon(t)$ is self-adjoint and non-negative. Hence, there exists a self-adjoint, non-negative, bounded, and $\mathbb{R}$-linear operator $A_\varepsilon(t)^{\frac{1}{2}}$, such that
\begin{equation*}
A_\varepsilon(t)=A_\varepsilon(t)^{\frac{1}{2}}A_\varepsilon(t)^{\frac{1}{2}}.
\end{equation*}
Using this square root operator in the integral \eqref{Eq_Accretive_boundedFC_1}, we can estimate for every $v,w\in V$
\begin{align}
\Sc\langle f(T_\varepsilon)v,w\rangle&=\frac{1}{2\pi}\int_\mathbb{R}\Sc\big\langle A_\varepsilon(t)^{\frac{1}{2}}vf(tJ),A_\varepsilon(t)^{\frac{1}{2}}w\big\rangle dt \notag \\
&\leq\Vert f\Vert_{\infty,S_\frac{\pi}{2}}\frac{1}{2\pi}\int_\mathbb{R}\Vert A_\varepsilon(t)^{\frac{1}{2}}v\Vert\Vert A_\varepsilon(t)^{\frac{1}{2}}w\Vert dt \notag \\
&\leq\Vert f\Vert_{\infty,S_{\frac{\pi}{2}}}\frac{1}{2\pi}\bigg(\int_\mathbb{R}\Vert A_\varepsilon(t)^{\frac{1}{2}}v\Vert^2dt\bigg)^{\frac{1}{2}}\bigg(\int_\mathbb{R}\Vert A_\varepsilon(t)^{\frac{1}{2}}w\Vert^2dt\bigg)^{\frac{1}{2}} \notag \\
&=\Vert f\Vert_{\infty,S_{\frac{\pi}{2}}}\frac{1}{2\pi}\bigg(\int_\mathbb{R}\Sc\langle A_\varepsilon(t)v,v\rangle dt\bigg)^{\frac{1}{2}}\bigg(\int_\mathbb{R}\Sc\langle A_\varepsilon(t)w,w\rangle dt\bigg)^{\frac{1}{2}}. \label{Eq_Accretive_boundedFC_2}
\end{align}
Moreover, plugging the constant function $f(s)=1$ into the integral \eqref{Eq_Accretive_boundedFC_1}, gives the identity
\begin{equation*}
u=\frac{1}{2\pi}\int_\mathbb{R}A_\varepsilon(t)udt,\qquad u\in V.
\end{equation*}
This then reduces the right hand side of \eqref{Eq_Accretive_boundedFC_2} to
\begin{equation*}
\Sc\langle f(T_\varepsilon)v,w\rangle\leq\Vert f\Vert_{\infty,S_{\frac{\pi}{2}}}\Vert v\Vert\Vert w\Vert,\qquad v,w\in V.
\end{equation*}
Since this is true for every $v,w\in V$, we conclude
\begin{equation}\label{Eq_Accretive_boundedFC_3}
\Vert f(T_\varepsilon)\Vert\leq\Vert f\Vert_{\infty,S_{\frac{\pi}{2}}}\qquad f\in\mathcal{N}^\infty(S_\theta),\,\theta\in(\tfrac{\pi}{2},\pi).
\end{equation}
In the \textit{second step} we will prove \eqref{Eq_Accretive_boundedFC} for the operator $T$, but only for functions $f\in\mathcal{N}^0(S_\theta)$. By Lemma~\ref{lem_Teps}, the resolvents of $T_\varepsilon$ are for fixed $\varphi\in(\frac{\pi}{2},\theta)$, uniformly bounded by
\begin{equation*}
\Vert(\mathcal{I}^Rs-T_\varepsilon)^{-1}\Vert\leq\frac{C_\varphi}{|s|},\qquad s\in\mathbb{R}^{n+1}\setminus(S_\varphi\cup\{0\}).
\end{equation*}
Hence, by the dominated convergence theorem, we can carry the limit inside the integral
\begin{align*}
\lim\limits_{\varepsilon\rightarrow 0^+}f(T_\varepsilon)v&=\frac{1}{2\pi}\lim\limits_{\varepsilon\rightarrow 0^+}\int_{\partial S_\varphi\cap\mathbb{C}_J}(\mathcal{I}^Rs-T_\varepsilon)^{-1}vf(s)ds_J \\
&=\frac{1}{2\pi}\int_{\partial S_\varphi\cap\mathbb{C}_J}(\mathcal{I}^Rs-T)^{-1}vf(s)ds_J=f(T)v,\qquad v\in V.
\end{align*}
Together with the estimate \eqref{Eq_Accretive_boundedFC_3}, this limit shows the bound
\begin{equation}\label{Eq_Accretive_boundedFC_4}
\Vert f(T)v\Vert=\lim\limits_{\varepsilon\rightarrow 0^+}\Vert f(T_\varepsilon)v\Vert\leq\lim\limits_{\varepsilon\rightarrow 0^+}\Vert f(T_\varepsilon)\Vert\Vert v\Vert\leq\Vert f\Vert_{\infty,S_{\frac{\pi}{2}}}\Vert v\Vert,\qquad v\in V.
\end{equation}
In the \textit{third step}, we will generalize the result to $f\in\mathcal{N}^\infty(S_\theta)$. Let $r_n(s)$ be the regularizing sequence from \eqref{Eq_Coercive_10}. With the constant \eqref{Eq_Teps_3}, these functions are uniformly bounded by
\begin{equation*}
|r_n(s)|=\frac{n|s|}{|s+n||s+\frac{1}{n}|}=\frac{1}{|1+\frac{s}{n}||1+\frac{1}{sn}|}\leq c_{\pi-\theta}^2,\qquad s\in S_\theta.
\end{equation*}
Together with \eqref{Eq_Accretive_boundedFC_4}, we then obtain the estimates
\begin{equation}\label{Eq_Accretive_boundedFC_5}
\sup\limits_{n\in\mathbb{N}}\Vert r_nf\Vert_{\infty,S_\theta}\leq c_{\pi-\theta}^2\Vert f\Vert_{\infty,S_\theta},\qquad\text{and}\qquad\sup\limits_{n\in\mathbb{N}}\Vert(r_nf)(T)\Vert\leq c_{\pi-\theta}^2\Vert f\Vert_{\infty,S_\theta}.
\end{equation}
Hence, by \cite[Lemma~4.1]{CMS25}, there is $f(T)\in\mathcal{B}(V)$, and there converges
\begin{equation*}
\lim\limits_{n\rightarrow\infty}(r_nf)(T)v=f(T)v,\qquad v\in V.
\end{equation*}
Consequently, with \eqref{Eq_Accretive_boundedFC_4} and the fact that $c_{\frac{\pi}{2}}=1$ in \eqref{Eq_Accretive_boundedFC_5}, we can estimate
\begin{equation*}
\Vert f(T)v\Vert=\lim\limits_{n\rightarrow\infty}\Vert(r_nf)(T)v\Vert\leq\lim\limits_{n\rightarrow\infty}\Vert r_nf\Vert_{\infty,S_{\frac{\pi}{2}}}\Vert v\Vert\leq\Vert f\Vert_{\infty,S_{\frac{\pi}{2}}}\Vert v\Vert,\qquad v\in V.
\end{equation*}
Since this is true for every $v\in V$, we have proven \eqref{Eq_Accretive_boundedFC} for every $f\in\mathcal{N}^\infty(S_\theta)$, $\theta\in(\frac{\pi}{2},\pi)$.
\end{proof}

In our application of Theorem~\ref{thm_Accretive_boundedFC} to the gradient operator $\nabla_a$ in Theorem~\ref{thm_Bounded_FC_gradient}, the gradient itself is not $m$-accretive, only its square $\nabla_a^2$ is. Hence we need the following corollary.

\begin{cor}\label{cor_T2_accretive}
Let $T\in\mathcal{K}(V)$ be injective, and bisectorial of angle $\omega\in(0,\frac{\pi}{2})$, such that $T^2$ is $m$-accretive. Then $T$ has a bounded $H^\infty$-functional calculus.
\end{cor}

\begin{proof}
Since $T^2$ is $m$-accretive, it is $\frac{\pi}{2}$-sectorial by Theorem~\ref{thm_Accretive_boundedFC}, and satisfies the estimate
\begin{equation*}
\Vert f(T^2)\Vert\leq\Vert f\Vert_{\infty,\frac{\pi}{2}},\qquad f\in\mathcal{N}^\infty(S_{2\theta}),\,\theta\in(\tfrac{\pi}{4},\tfrac{\pi}{2}).
\end{equation*}
In particular, $T^2$ is sectorial of any angle $\omega'\in(\frac{\pi}{2},\pi)$ and admits a bounded $H^\infty$-functional calculus considered as an $\omega'$-sectorial operator. However, since $T^2$ is even sectorial of angle $2\omega$ by \cite[Proposition~3.26]{CMS26}, which may be smaller then $\frac{\pi}{2}$, it follows from Theorem~\ref{thm_Quadratic_estimates} that $T^2$ also admits a bounded $H^\infty$-functional calculus with respect to this angle $2\omega$. Finally, Theorem~\ref{thm_Bounded_FC_T2} tells us that in this case also $T$ has a bounded $H^\infty$-functional calculus, considered as a bisectorial operator of angle $\omega$.
\end{proof}

\section{Spectral projectors}\label{sec_Projectors}

Let us note that for any $\theta\in(0,\frac{\pi}{2})$, the double sector $D_\theta$ in \eqref{Eq_Domega} can be decomposed into two disconnected sectors $S_\theta$ and $-S_\theta$ from \eqref{Eq_Somega}. Moreover, for every function $f:D_\theta\rightarrow\mathbb{R}^{n+1}$, we define the restrictions to the respective sectors \medskip

\begin{minipage}{0.29\textwidth}
\begin{center}
\begin{tikzpicture}
\fill[black!20] (0,0)--(1.81,0.85) arc (25:-25:2)--(0,0)--(-1.81,-0.85) arc (205:155:2);
\draw (-1.81,0.85)--(1.81,-0.85);
\draw (-1.81,-0.85)--(1.81,0.85);
\draw[->] (-2.2,0)--(2.3,0);
\draw[->] (0,-0.8)--(0,0.8);
\draw (1,0) arc (0:25:1) (0.8,-0.07) node[anchor=south] {\small{$\theta$}};
\draw (-1.4,-0.05) node[anchor=south] {$-S_\theta$};
\draw (1.4,-0.05) node[anchor=south] {$S_\theta$};
\end{tikzpicture}
\end{center}
\end{minipage}
\begin{minipage}{0.7\textwidth}
\begin{equation}\label{Eq_fpm}
f_+(s):=\begin{cases} f(s), & s\in S_\theta, \\ 0, & s\in-S_\theta, \end{cases}\quad f_-(s):=\begin{cases} 0, & s\in S_\theta, \\ f(s), & s\in-S_\theta. \end{cases}
\end{equation}
\end{minipage}

\medskip If $f$ is intrinsic, then also $f_\pm$ are intrinsic. Moreover, if $f\in\mathcal{N}^0(D_\theta)$, then also $f_\pm\in\mathcal{N}^0(D_\theta)$, and for every bisectorial operator $T\in\mathcal{K}(V)$ of angle $\omega\in(0,\frac{\pi}{2})$, the $\omega$-functional calculus is compatible with this decomposition in the sense
\begin{equation*}
f(T)=f_+(T)+f_-(T).
\end{equation*}
If moreover $\sigma_S(T)\subseteq\pm\overline{S_\omega}$, then
\begin{equation*}
f(T)v=f_\pm(T)v=\frac{1}{2\pi}\int_{\partial(\pm S_\varphi)\cap\mathbb{C}_J}(\mathcal{I}^Rs-T)^{-1}vf(s)ds_J.
\end{equation*}
In other words, the bisectorial functional calculus reduces in this case to the sectorial functional calculus, where the integration path only surrounds one of the two sectors. \medskip

Now we discuss if it is possible to decompose a bisectorial operator $T$ into a sum of two operators, corresponding to the decomposition of the $S$-spectrum into the two parts of the double sector. To do so, let us consider the characteristic functions
\begin{equation*}
\mathds{1}_+(s):=\begin{cases} 1 & \Sc(s)>0, \\ 0 & \Sc(s)<0, \end{cases}\qquad\text{and}\qquad\mathds{1}_-(s):=\begin{cases} 0, & \Sc(s)>0, \\ 1, & \Sc(s)<0. \end{cases}
\end{equation*}
These functions are constant and hence intrinsic on both half spaces $\mathbb{R}^{n+1}_+$ and $\mathbb{R}^{n+1}_-$. However, they do not decay at $0$ and at $\infty$, which means we cannot use the $\omega$-functional calculus to define the projectors $\mathds{1}_\pm(T)$. Moreover, since these functions admit a jump at $s=0$ we can also not use something like the extended $\omega$-functional calculus \cite[Definition~4.3]{MS24}. Consequently, we need the full $H^\infty$-functional calculus from Definition~\ref{defi_omega_Hinfty}~ii'), which causes the problem that $\mathds{1}_\pm(T)$ will become unbounded in general, and the upcoming theory will no longer be valid. Hence we need to assume that the operator $T$ admits a bounded $H^\infty$-functional calculus, in order to ensure that $\mathds{1}_\pm(T)$ are bounded.

\begin{defi}\label{defi_P}
Let $T\in\mathcal{K}(V)$ be injective, bisectorial of angle $\omega\in(0,\frac{\pi}{2})$, and have a bounded $H^\infty$-functional calculus. Then we define the \textit{spectral projectors}
\begin{equation*}
P_+:=\mathds{1}_+(T),\qquad\text{and}\qquad P_-:=\mathds{1}_-(T).
\end{equation*}
Both operators are defined via the $H^\infty$-functional calculus from Definition~\ref{defi_omega_Hinfty}~ii'), and by assumption they are bounded, everywhere defined operators $P_+,P_-\in\mathcal{B}(V)$.
\end{defi}

\begin{lem}\label{lem_P_properties}
Let $T$ be injective, bisectorial of angle $\omega\in(0,\frac{\pi}{2})$, and have a bounded $H^\infty$-functional calculus. Then there holds
\begin{enumerate}
\item[a)] $P_+^2 = P_+$,\quad and\quad $P_-^2=P_-$,
\item[b)] $P_+P_-=P_-P_+=0$,
\item[c)] $P_++P_-=1$.
\end{enumerate}
\end{lem}

\begin{proof}
a), b)\;\;By the product rule of the $H^\infty$-functional calculus \cite[Theorem~5.7]{MS24}, there is
\begin{equation*}
P_+^2=\mathds{1}_+(T)\mathds{1}_+(T)\subseteq(\mathds{1}_+\mathds{1}_+)(T)=\mathds{1}_+(T)=P_+.
\end{equation*}
However, since $P_+\in\mathcal{B}(V)$ is everywhere defined, this operator inclusion holds with equality. Analogous computations also show that $P_-^2=P_-$, $P_+P_-=0$, and $P_-P_+=0$. \medskip

c)\;\;From the linearity of the $H^\infty$-functional calculus \cite[Lemma~5.5~(ii)]{MS24}, we know that
\begin{equation*}
P_++P_-=\mathds{1}_+(T)+\mathds{1}_-(T)\subseteq(\mathds{1}_++\mathds{1}_-)(T)=1(T)=1.
\end{equation*}
Since $P_+,P_-\in\mathcal{B}(V)$ are everywhere defined, this operator inclusion holds with equality.
\end{proof}

\begin{defi}\label{defi_V}
Let $T$ be injective, bisectorial of angle $\omega\in(0,\frac{\pi}{2})$, and have a bounded $H^\infty$-functional calculus. Then we define the spaces
\begin{equation*}
V_+:=\ran P_+,\qquad\text{and}\qquad V_-:=\ran P_-.
\end{equation*}
\end{defi}

\begin{lem}\label{lem_V_decomposition}
Let $T$ be injective, bisectorial of angle $\omega\in(0,\frac{\pi}{2})$ and have a bounded $H^\infty$-functional calculus. Then $V_+$ and $V_-$ are closed, right linear subspaces of $V$, and the whole space admits the direct sum decomposition
\begin{equation}\label{Eq_V_decomposition}
V=V_+\overset{\bullet}{+}V_-.
\end{equation}
\end{lem}

\begin{proof}
The spaces $V_+$ and $V_-$ are right linear, as they are the ranges of the right linear operators $P_+$ and $P_-$. The closedness of $V_\pm$ is an immediate consequence of the boundedness of $P_\pm$ and the property $P_\pm^2=P_\pm$ from Lemma~\ref{lem_P_properties}~a). To prove the decomposition \eqref{Eq_V_decomposition}, let $v\in V$. By Lemma~\ref{lem_P_properties}~c), we can write
\begin{equation*}
v=P_+v+P_-v,
\end{equation*}
which shows that $v\in V_++V_-$. In order to show that $V_+\cap V_-=\{0\}$, let $v\in V_+\cap V_-$. Since $v\in V_+=\ran P_+$, there is $P_-v=0$ due to Lemma~\ref{lem_P_properties}~b). Analogously, there is $P_+v=0$ and hence $v=P_+v+P_-v=0$.
\end{proof}

Now that we have defined the spectral projectors $P_\pm$ and the corresponding subspaces $V_\pm$, we want to restrict operators to these subspaces. However, this can only be done if the subspaces are invariant under the action of the operator. This property is formulated in three equivalent ways in the following lemma.

\begin{lem}[Invariant subspaces]\label{lem_Invariant_subspace}
Let $T\in\mathcal{K}(V)$ be injective, bisectorial of angle $\omega\in(0,\frac{\pi}{2})$ and have a bounded $H^\infty$-functional calculus. Then for any $S\in\mathcal{K}(V)$ the following statements are equivalent
\begin{enumerate}
\item[i)] $\ran(S\rest{V_\pm})\subseteq V_\pm$,\quad and\quad $\ran(P_\pm\rest{\dom S})\subseteq\dom S$, \medskip
\item[ii)] $SP_+\rest{\dom S}\,=P_+S$, \medskip
\item[iii)] $SP_-\rest{\dom S}\,=P_-S$.
\end{enumerate}
\end{lem}

\begin{proof}
We will only prove the equivalence i) $\Leftrightarrow$ ii), while i) $\Leftrightarrow$ iii) then follows analogously. For the implication i) $\Rightarrow$ ii), let us first prove the inclusion $\subseteq$. Let $v\in\dom S$ with $P_+v\in\dom S$. If we decompose $v=P_+v+P_-v$, possible by Lemma~\ref{lem_P_properties}~c), we also get $P_-v\in\dom S$. Hence, we can write
\begin{equation}\label{Eq_Invariant_subspace_1}
SP_+v=Sv-SP_-v.
\end{equation}
By assumption we know that $SP_+v\in V_+$ and $SP_-v\in V_-$. I.e., if we apply the operator $P_+$ onto the equation \eqref{Eq_Invariant_subspace_1}, we get with Lemma~\ref{lem_P_properties}~a) \& b) the desired
\begin{equation*}
SP_+v=P_+Sv.
\end{equation*}
For the inverse inclusion $\supseteq$ let $v\in\dom S$. Then by i) there is also $P_+v\in\dom S$, and consequently $v\in\dom(SP_+\rest{\dom S})$. \medskip

For the implication ii) $\Rightarrow$ i), let first $v\in\ran(S\rest{V_+})$, i.e., $v=Su$ for some $u\in V_+\cap\dom S$. Consequently, there is $P_+u=u\in\dom S$, and we get from our assumption ii), that
\begin{equation*}
v=Su=SP_+u=P_+Su\in\ran P_+=V_+.
\end{equation*}
For the second inclusion let $v\in\ran(S\rest{V_-})$, i.e., $v=Su$ for some $u\in V_-\cap\dom S$. Since there is also $P_+u=0\in\dom S$, we get from ii), that $0=SP_+u=P_+Su$. This means
\begin{equation*}
v=Su=P_-Su\in\ran P_-=V_-.
\end{equation*}
Next, let $v\in\dom S$. This means, $v$ is in the domain of the right hand side of ii). Hence, $v$ is also in the domain of the left hand side of ii), which means that $P_+v\in\dom S$. Moreover, since we can write $P_-v=v-P_+v$ by Lemma~\ref{lem_P_properties}~c), there is also $P_-v\in\dom S$.
\end{proof}

\begin{defi}\label{defi_Spm}
Let $T$ be injective, bisectorial of angle $\omega\in(0,\frac{\pi}{2})$ and have a bounded $H^\infty$-functional calculus. Furthermore, let $S\in\mathcal{K}(V)$ satisfy the equivalent conditions i)--iii) of Lemma~\ref{lem_Invariant_subspace}. In this case, we define the \text{restricted operators}
\begin{equation*}
S_\pm:=S\rest{V_\pm}:V_\pm\rightarrow V_\pm,\qquad\text{with}\quad\dom S_\pm:=V_\pm\cap\dom S.
\end{equation*}
\end{defi}

\begin{prop}\label{prop_dom_ran_Spm}
Let $T$ be injective, bisectorial of angle $\omega\in(0,\frac{\pi}{2})$ and have a bounded $H^\infty$-functional calculus. Furthermore, let $S\in\mathcal{K}(V)$ satisfy the equivalent conditions i)--iii) of Lemma~\ref{lem_Invariant_subspace}. Then, there is
\begin{enumerate}
\item[a)] $\ran S_\pm=\ran S\cap V_\pm$, \medskip
\item[b)] $\dom S_\pm=\ran(P_\pm\rest{\dom S})$.
\end{enumerate}
\end{prop}

\begin{proof}
In both points, we will only show the \grqq$+$\grqq-part, the \grqq$-$\grqq-part follows analogously. \medskip

a)\;\;The inclusion $\subseteq$ follows immediately from Lemma~\ref{lem_Invariant_subspace}~i). For the inclusion $\supseteq$, let $v\in\ran S\cap V_+$. Hence, there is $v=Su$ for some $u\in\dom S$. By Lemma~\ref{lem_Invariant_subspace}~i), we know that also $P_+u\in\dom S$. From the commutation in Lemma~\ref{lem_Invariant_subspace}~ii) and from $v\in V_+$, it then follows that
\begin{equation*}
v=P_+v=P_+Su=SP_+u\in\ran(S\rest{V_+})=\ran(S_+).
\end{equation*}
b)\;\;For the inclusion $\supseteq$, let $v\in\dom S$. Then there is $P_+v\in V_+$, but also $P_+v\in\dom S$ by the operator inclusion in Lemma~\ref{lem_Invariant_subspace}~ii). Hence there is $P_+v\in\dom(S\rest{V_+})=\dom S_+$. \medskip

For the inclusion $\subseteq$, let $v\in\dom S_+$, i.e., $v\in\dom(S)$ and $v\in V_+$. Then there clearly is $v=P_+v\in\ran(P_+\rest{\dom S})$.
\end{proof}

Until now, operators $S\in\mathcal{K}(V)$ for which $S_\pm$ are defined, are given abstractly by the conditions in Lemma~\ref{lem_Invariant_subspace}. The next corollary shows that every operator $f(T)$, which comes from the $H^\infty$-functional calculus of the underlying operator $T$, is one of them.

\begin{cor}\label{cor_fTpm}
Let $T$ be injective, bisectorial of angle $\omega\in(0,\frac{\pi}{2})$ and have a bounded $H^\infty$-functional calculus. Then, for every $f\in\mathcal{N}^\poly(D_\theta)$, $\theta\in(\omega,\frac{\pi}{2})$, the operator $f(T)$ satisfies
\begin{equation}\label{Eq_fTpm}
f(T)P_\pm\rest{\dom f(T)}=P_\pm f(T).
\end{equation}
In particular, $f(T)$ satisfies the conditions i)--iii) of Lemma~\ref{lem_Invariant_subspace}. Moreover, there holds
\begin{enumerate}
\item[a)] $\ran f(T)_\pm=\ran f(T)\cap V_\pm$, \medskip
\item[b)] $\dom f(T)_\pm=\ran(P_\pm\rest{\dom f(T)})$.
\end{enumerate}
\end{cor}

\begin{proof}
We only have to prove \eqref{Eq_fTpm}. The equivalent conditions i)--iii) of Lemma~\ref{lem_Invariant_subspace} are then satisfied clearly, and the properties a) and b) follow from Proposition~\ref{prop_dom_ran_Spm}. Also, we will only consider the \grqq$+$\grqq-part in \eqref{Eq_fTpm}, the \grqq$-$\grqq-part then follows analogously. \medskip

From the product rule of the $H^\infty$-functional calculus \cite[Theorem~5.7]{MS24}, we know that
\begin{equation*}
P_+f(T)\subseteq(\mathds{1}_+f)(T)\qquad\text{and}\qquad f(T)P_+\subseteq(f\mathds{1}_+)(T),
\end{equation*}
have the same value on their respective domains, and their domains are given by
\begin{align*}
\dom(P_+f(T))&=\dom(\mathds{1}_+f)(T)\cap\dom f(T), \\
\dom(f(T)P_+)&=\dom(f\mathds{1}_+)(T)\cap\dom\mathds{1}_+(T)=\dom(f\mathds{1}_+)(T).
\end{align*}
Combining these two domain identities, gives
\begin{equation*}
\dom(P_+f(T))=\dom(f(T)P_+)\cap\dom f(T),
\end{equation*}
ans consequently, we get $f(T)P_+\rest{\dom f(T)}=P_+f(T)$.
\end{proof}

Choosing the special function $f(s)=s$ in Corollary~\ref{cor_fTpm}, we know that $T$ satisfies the equivalent conditions i)--iii) of Lemma~\ref{lem_Invariant_subspace}. Hence, the restricted operators $T_+$ and $T_-$ are defined in the sense of Definition~\ref{defi_Spm}. The next theorem now closer investigates them in terms of their $S$-spectra.

\begin{thm}\label{thm_Spectrum_restriction}
Let $T$ be injective, bisectorial of angle $\omega\in(0,\frac{\pi}{2})$ and have a bounded $H^\infty$-functional calculus. Then, using the half spaces $\mathbb{R}^{n+1}_\pm$ from \eqref{Eq_Accretive_5}, there is
\begin{equation}\label{Eq_Spectrum_restriction}
\sigma_S(T_\pm)\setminus\{0\}=\sigma_S(T)\cap\mathbb{R}^{n+1}_\pm.
\end{equation}
For $s=0$ on the other hand, there is
\begin{equation}\label{Eq_Zero_in_the_spectrum}
0\in\sigma_S(T)\quad\Leftrightarrow\quad 0\in\sigma_S(T_+)\cup\sigma_S(T_-).
\end{equation}
\end{thm}

\begin{proof}
In \eqref{Eq_Spectrum_restriction} we will only verify the equality for $T_+$. First of all, let us note that there is
\begin{equation}\label{Eq_Spectrum_restriction_4}
\dom T_+^2=V_+\cap\dom T^2\qquad\text{and}\qquad Q_s[T_+]v=Q_s[T]v,\quad v\in\dom T_+^2.
\end{equation}
The inclusion $\subseteq$ in \eqref{Eq_Spectrum_restriction} will be done in two parts. In the \textit{first part}, we will show that
\begin{equation}\label{Eq_Spectrum_restriction_1}
\sigma_S(T_+)\subseteq\overline{S_\omega},
\end{equation}
i.e., for $s\in\mathbb{R}^{n+1}\setminus\overline{S_\omega}$ we have to verify that the operator $Q_s[T_+]:\dom T_+^2\rightarrow V_+$ is bijective. With the function $Q_s(q):=q^2-2s_0q+|s|^2$, let us define
\begin{equation*}
h_s(q):=\begin{cases} Q_s(q)^{-1}, & q\in S_\omega, \\ 0, & q\in-S_\omega. \end{cases}.
\end{equation*}
Then $h_s\in\mathcal{N}^\infty(D_\omega)$, and $h_s(T)$ is a bounded operator, defined via the $H^\infty$-functional calculus. Since the operator polynomial $Q_s[T]=Q_s(T)$ coincides with the $H^\infty$-functional calculus by \cite[Theorem~5.9]{MS24}, we can use the product rule \cite[Theorem~5.7]{MS24}, to get
\begin{align*}
h_s(T)Q_s[T]&\subseteq(h_sQ_s)(T)=\mathds{1}_+(T)=P_+, \\
Q_s[T]h_s(T)&=(Q_sh_s)(T)=\mathds{1}_+(T)=P_+.
\end{align*}
These two identities, together with \eqref{Eq_Spectrum_restriction_4}, then turn into
\begin{align*}
h_s(T)Q_s[T_+]v&=h_s(T)Q_s[T]v=P_+v=v,\qquad v\in\dom T_+^2, \\
Q_s[T_+]h_s(T)v&=Q_s[T]h_s(T)v=P_+v=v,\qquad v\in V_+,
\end{align*}
where in the first equation of the second line we used that $h_s(T)v\in V_+$ by Corollary~\ref{cor_fTpm}~a). These identities now show that $Q_s[T_+]$ is bijective, i.e., $s\in\rho_S(T_+)$, and we have verified \eqref{Eq_Spectrum_restriction_1}. \medskip

In the \textit{second part} of the inclusion $\subseteq$ of \eqref{Eq_Spectrum_restriction}, we will show
\begin{equation}\label{Eq_Spectrum_restriction_5}
\sigma_S(T_+)\subseteq\sigma_S(T).
\end{equation}
To do so, let $s\in\rho_S(T)$, i.e., $Q_s[T]$ is bijective. Then we obtain the two identities
\begin{align*}
Q_s[T]^{-1}Q_s[T_+]v&=Q_s[T]^{-1}Q_s[T]v=v,\qquad v\in\dom T_+^2, \\
Q_s[T_+]Q_s[T]^{-1}v&=Q_s[T]Q_s[T]^{-1}v=v,\qquad v\in V_+,
\end{align*}
where in the first equation of the second line we used that $Q_s[T]^{-1}v\in V_+$ by \eqref{Eq_Spectrum_restriction_4}. These two identities now show that $Q_s[T_+]$ is bijective, i.e., $s\in\rho_S(T^+)$, and we have verified \eqref{Eq_Spectrum_restriction_5}. \medskip

Combining now the identities \eqref{Eq_Spectrum_restriction_1} and \eqref{Eq_Spectrum_restriction_5}, proves the first inclusion in \eqref{Eq_Spectrum_restriction}, namely
\begin{equation*}
\sigma_S(T_+)\setminus\{0\}\subseteq\sigma_S(T)\cap\overline{S_\omega}\setminus\{0\}\subseteq\sigma_S(T)\cap\mathbb{R}^{n+1}_+.
\end{equation*}
For the inclusion $\supseteq$ in \eqref{Eq_Spectrum_restriction}, let $s\in\rho_S(T_+)\setminus(-\overline{S_\omega})$. From the inclusion \eqref{Eq_Spectrum_restriction_1}, for $T_-$ instead of $T_+$, we also know that $s\in\rho_S(T_-)$. This means, $Q_s[T_+]$ and $Q_s[T_-]$ are both bijective. Since by \eqref{Eq_Spectrum_restriction_4}, the commutation \eqref{Eq_fTpm} writes as
\begin{equation*}
P_\pm Q_s[T]\subseteq Q_s[T]P_\pm=Q_s[T_\pm]P_\pm,
\end{equation*}
it follows that
\begin{subequations}
\begin{align}
\big(Q_s[T_+]^{-1}P_++Q_s[T_-]^{-1}P_-\big)Q_s[T]v&=P_+v+P_-v=v,\qquad v\in\dom T^2, \label{Eq_Spectrum_restriction_3} \\
Q_s[T]\big(Q_s[T_+]^{-1}P_++Q_s[T_-]^{-1}P_-\big)v&=P_+v+P_-v=v,\qquad v\in V, \label{Eq_Spectrum_restriction_6}
\end{align}
\end{subequations}
These two identities now show that $Q_s[T]$ is bijective, i.e., $s\in\rho_S(T)$. Hence we have proven
\begin{equation}\label{Eq_Spectrum_restriction_2}
\rho_S(T_+)\setminus(-\overline{S_\omega})\subseteq\rho_S(T).
\end{equation}
Since we already know by \eqref{Eq_Spectrum_restriction_1} that there is $\sigma_S(T_+)\subseteq\overline{S_\omega}$, we conclude from \eqref{Eq_Spectrum_restriction_2} that
\begin{equation*}
\sigma_S(T)\cap\mathbb{R}^{n+1}_+\subseteq\sigma_S(T_+)\cap\mathbb{R}^{n+1}_+=\sigma_S(T_+)\setminus\{0\}.
\end{equation*}
Now that we have proven \eqref{Eq_Spectrum_restriction}, we turn our attention to the equivalence \eqref{Eq_Zero_in_the_spectrum}.  The implication $\Leftarrow$ follows already from \eqref{Eq_Spectrum_restriction_5} and the according inclusion for $T_-$. For the implication $\Rightarrow$ on the other hand assume that $0\in\rho_S(T_+)\cap\rho_S(T_-)$, i.e., $Q_0[T_+]$ and $Q_0[T_-]$ are invertible. Then the same two identities as in \eqref{Eq_Spectrum_restriction_3} and \eqref{Eq_Spectrum_restriction_6} prove that $Q_0[T]$ is bijective, which is a contradiction to $0\in\sigma_S(T)$.
\end{proof}

\begin{cor}
Let $T$ be injective, bisectorial of angle $\omega\in(0,\frac{\pi}{2})$ and have a bounded $H^\infty$-functional calculus. Then
\begin{equation}\label{Eq_Resolvent_set_Restriction}
\rho_S(T)=\rho_S(T_+)\cap\rho_S(T_-),
\end{equation}
and for every $s\in\rho_S(T)$, the resolvent operator decomposes into
\begin{equation}\label{Eq_Resolvent_decomposition}
(\mathcal{I}^Rs-T)^{-1}=(\mathcal{I}^Rs-T_+)^{-1}P_++(\mathcal{I}^Rs-T_-)^{-1}P_-.
\end{equation}
\end{cor}

\begin{proof}
From the two identities in \eqref{Eq_Spectrum_restriction}, it in particular follows that
\begin{equation*}
\big(\rho_S(T_+)\cap\rho_S(T_-)\big)\cup\{0\}=\rho_S(T)\cup\big\{s\in\mathbb{R}^{n+1}\;\big|\;s_0=0\big\}=\rho_S(T)\cup\{0\},
\end{equation*}
where in the second equation we used that $\{s\in\mathbb{R}^{n+1}\;|\;s_0=0\}\setminus\{0\}\subseteq\rho_S(T)$, because $T$ is bisectorial. Using also the equivalence
\begin{equation*}
0\in\rho_S(T)\quad\Leftrightarrow\quad 0\in\rho_S(T_+)\cap\rho_S(T_-),
\end{equation*}
from \eqref{Eq_Zero_in_the_spectrum}, we conclude the equality \eqref{Eq_Resolvent_set_Restriction}. Moreover, the decomposition \eqref{Eq_Resolvent_decomposition} follows immediately from
\begin{equation*}
(\mathcal{I}^Rs-T)v=(\mathcal{I}^Rs-T_\pm)v,\qquad v\in\dom T_\pm. \qedhere
\end{equation*}
\end{proof}

\begin{thm}
Let $T$ be injective, bisectorial of angle $\omega\in(0,\frac{\pi}{2})$ and have a bounded $H^\infty$-functional calculus. Then, for every $f\in\mathcal{N}^\poly(D_\theta)$, $\theta\in(\omega,\frac{\pi}{2})$, there is
\begin{equation}\label{Eq_SFC_restriction}
f(T)_\pm=f_\pm(T)_\pm=f(T_\pm).
\end{equation}
The functions $f_\pm$ are the ones in \eqref{Eq_fpm}. The first two expressions $f(T)_\pm$ and $f_\pm(T)_\pm$ are understood as functional calculi of the operator $T$ in the whole space $V$, which are then restricted to the subspace $V_\pm$. On the other hand, $f(T_\pm)$ is the functional calculus of the restricted operator $T_\pm$, already in the subspace $V_\pm$.
\end{thm}

\begin{proof}
We will only prove the \grqq$+$\grqq-part of the statement, the \grqq$-$\grqq -part is analog. If we restrict the identity \eqref{Eq_fTpm} to the subspace $V_+$, we get
\begin{equation*}
f(T)\rest{V_+}=f_+(T)\rest{V_+}.
\end{equation*}
For the second equation in \eqref{Eq_SFC_restriction}, we first consider $f\in\mathcal{N}^0(D_\theta)$. Then, for every $v\in V_+$, it follows from \eqref{Eq_Resolvent_decomposition} that the $\omega$-functional calculus is given by
\begin{equation*}
f(T)v=\frac{1}{2\pi}\int_{\partial D_\varphi\cap\mathbb{C}_J}(\mathcal{I}^Rs-T)^{-1}vds_Jf(s)=\frac{1}{2\pi}\int_{\partial D_\varphi\cap\mathbb{C}_J}(\mathcal{I}^Rs-T_+)^{-1}vds_Jf(s)=f(T_+)v.
\end{equation*}
For the extension to functions $f\in\mathcal{N}^\poly(D_\theta)$, let $e\in\mathcal{N}^0(D_\theta)$ be a regularizer function according to Definition~\ref{defi_omega_Hinfty}~ii'). Then we get
\begin{equation*}
f(T_+)=e(T_+)^{-1}(ef)(T_+)=(e(T)\rest{V_+})^{-1}(ef)(T)\rest{V_+}=e(T)^{-1}(ef)(T)\rest{V_+}=f(T)\rest{V_+}. \qedhere
\end{equation*}
\end{proof}

\section{Projectors of the gradient}\label{sec_Gradient}

The main application of the projectors from Section~\ref{sec_Projectors} will be the gradient operator with non-constant coefficients in $n\geq 3$ dimensions, first considered in the article \cite{CMS24},
\begin{equation}\label{Eq_Gradient}
\nabla_a:=\sum_{i=1}^ne_ia_i(x)\frac{\partial}{\partial x_i},\qquad\text{with }\dom\nabla_a:=H^1(\mathbb{R}^n,\mathbb{R}_n).
\end{equation}
This is an operator in the Clifford module $L^2(\mathbb{R}^n,\mathbb{R}_n)$  of square integrable functions with values in $\mathbb{R}_n$. For the coefficients $a_1,\dots,a_n$, we assume one of the following two assumptions:

\begin{itemize}
\item[I)] If the functions $a_i$ depend on all variables $x_1,\dots,x_n$, we assume that
\begin{subequations}
\begin{align}
m_a:=&\min\limits_{i\in\{1,\dots,n\}}\inf\limits_{x\in\mathbb{R}^n}a_i(x)>0, \label{Eq_ma} \\
M_a:=&\Big(\sum\nolimits_{i=1}^n\Vert a_i\Vert_{L^\infty}^2\Big)^{\frac{1}{2}}<\infty, \label{Eq_Ma} \\
M_a':=&\Big(\sum\nolimits_{i,j=1}^n\Big\Vert a_j\frac{\partial a_i}{\partial x_j}\Big\Vert^2_{L^n}\Big)^{\frac{1}{2}}<\infty, \label{Eq_Maprime} \\
M_a'':=&\Big(\sum\nolimits_{i,j=1}^n\Big\Vert\frac{\partial a_i}{\partial x_j}\Big\Vert_{L^\infty}^2\Big)^{\frac{1}{2}}<\infty. \label{Eq_Maprimeprime}
\end{align}
\end{subequations}
Moreover, with the constant $C_S$ from the Sobolev embedding $H^1(\mathbb{R}^n)\subseteq L^{\frac{2n}{n-2}}(\mathbb{R}^n)$, we assume these upper bounds satisfy
\begin{equation*}
m_a^2>C_SM_a'.
\end{equation*}

\item[II)] If $a_i(x)=a_i(x_i)$ only depends on the variable $x_i$, for every $i\in\{1,\dots,n\}$, we only assume the bounds \eqref{Eq_ma}, \eqref{Eq_Ma}, and \eqref{Eq_Maprimeprime}.
\end{itemize}

\begin{thm}\label{thm_Bounded_FC_gradient}
The gradient $\nabla_a$ in \eqref{Eq_Gradient} is injective and bisectorial of angle $\omega\in(0,\frac{\pi}{2})$, where
\begin{equation*}
\omega:=\begin{cases} \arctan\Big(\sqrt{\frac{M_a^2}{m_a^2-C_SM_a'}-1}\Big), & \text{in case I)}, \\ \arctan\left(\sqrt{n-1}\right), & \text{in case II)}. \end{cases}
\end{equation*}
Moreover, $\nabla_a$ admits a bounded $H^\infty$-functional calculus.
\end{thm}

\begin{proof}
The fact that $\nabla_a$ is injective and bisectorial of angle $\omega$ is already proven in \cite[Theorem~4.7]{CMS26}. In order to show that $\nabla_a$ admits a bounded $H^\infty$-functional calculus, we will prove that the squared operator $\nabla_a^2$ is $m$-accretive and then use Corollary~\ref{cor_T2_accretive}. \medskip

For the condition a) of Definition~\ref{defi_Accretive}, we note that there is $\nabla_a^2=Q_0[\nabla_a]$. In case I), Using the weak formulation \cite[Equation~(1.8)]{Gradient}, we are able to write
\begin{equation*}
\langle\nabla_a^2v,v\rangle_{L^2}=q_0(v,v),\qquad v\in\dom\nabla_a^2.
\end{equation*}
Then, in \cite[Equation~(3.7)]{Gradient} it is proven that the scalar part of the form $q_0$ on the right hand side is bounded from below by
\begin{equation*}
\Sc\langle\nabla_a^2v,v\rangle_{L^2}\geq(m_a^2-C_SM_a')\Vert\nabla v\Vert_{L^2}^2\geq 0,\qquad v\in\dom\nabla_a^2.
\end{equation*}
In case II), we know from \cite[Theorem 4.3]{CMS26}, that the transformation
\begin{equation*}
y_i:=\int_0^{x_i}\frac{1}{a_i(t)}dt,\qquad x_i\in\mathbb{R},
\end{equation*}
induces a change of variable, such that
\begin{equation*}
a_i(x_i)\frac{\partial}{\partial x_i}=\frac{\partial}{\partial y_i}.
\end{equation*}
In particular, we have
\begin{equation*}
\nabla_a^2=-\sum_{i=1}^n\frac{\partial^2}{\partial y_i^2}=\nabla^2,
\end{equation*}
where in the second equality, the gradient with no coefficients $\nabla$ must be understood with respect to the variables $y_i$. For every $u\in\dom\nabla_a^2$, we now have
\begin{equation*}
\Sc\langle\nabla_a^2v,v\rangle_{L^2_{dx}}=\Sc\langle\nabla^2v,a_1\dots a_nv\rangle_{L^2_{dy}}\geq m_a^n\Sc\langle\nabla^2v,v\rangle_{L^2_{dy}}=m_a^n\Vert\nabla v\Vert_{L^2_{dy}}^2\geq 0.
\end{equation*}
For the condition b) of Definition~\ref{defi_Accretive}, we know by \cite[Proposition~3.26]{CMS26}, that $\nabla_a^2$ is sectorial of angle $2\omega$. This in particular means that $-1\in\rho_S(\nabla_a^2)$ and hence $\ran(1+\nabla_a^2)=L^2(\mathbb{R}^n,\mathbb{R}_n)$. \medskip

Now that we have verified that the operator $\nabla_a^2$ is $m$-accretive, the bounded $H^\infty$-functional calculus of $\nabla_a$ follows from Corollary~\ref{cor_T2_accretive}.
\end{proof}

In Theorem~~\ref{thm_Bounded_FC_gradient} we have proven that $\nabla_a$ is an injective, bisectorial operator which admits a bounded $H^\infty$-functional calculus. Hence, the projectors $P_\pm$ from Definition~\ref{defi_P}, as well as the corresponding subspaces $V_\pm$ in Definition~\ref{defi_V}, are well defined. We will now investigate how these objects look explicitly in the case of the gradient with constant coefficients $a_i(x)=a_i$. \medskip

The tool of choice will be the Fourier transform of Clifford valued functions. More precisely, we embed the Clifford algebra $\mathbb{R}_n$ into the complex Clifford algebra $\mathbb{C}_n$ by introducing an artificial imaginary unit $\mathbf{i}$, which commutes with all the imaginary units $e_1,\dots,e_n$ of the Clifford algebra $\mathbb{R}_n$. Then we define for every $u=\sum_{A\in\mathcal{A}}e_Au_A\in L^2(\mathbb{R}^n,\mathbb{C}_n)$, with $u_A\in L^2(\mathbb{R}^n,\mathbb{C})$, the Fourier transform componentwise as
\begin{equation*}
F[u](\xi):=\sum\limits_{A\in\mathcal{A}}e_AF[u_A](\xi) \in L^2(\mathbb{R}^n,\mathbb{C}_n),\qquad\xi\in\mathbb{R}^n.
\end{equation*}
Note that for the real Fourier transform, we use the convention
\begin{equation*}
F[u_A](\xi)=\frac{1}{(2\pi)^{\frac{n}{2}}}\int_{\mathbb{R}^n}e^{-\mathbf{i}\,x\cdot\xi}u_A(x)dx,\qquad u_A\in\mathcal{S}(\mathbb{R}^n,\mathbb{C}).
\end{equation*}
Let us start with the action of the projectors $P_\pm$ in Fourier space.

\begin{thm}\label{thm_Fourier_P}
The projectors $P_\pm$ of the gradient $\nabla_a$ with constant coefficients $a_i(x)=a_i>0$ admits the Fourier representation
\begin{equation*}
F[P_\pm u](\xi)=\frac{1}{2}\Big(1\pm\frac{\mathbf{i}\,\xi_a}{|\xi_a|}\Big)F[u](\xi),\qquad u\in L^2(\mathbb{R}^n,\mathbb{R}_n),\,\xi\in\mathbb{R}^n,
\end{equation*}
using the notion $\xi_a:=a_1\xi_1e_1+\dots+a_n\xi_ne_n$.
\end{thm}

\begin{proof}
For every Schwartz function $u\in\mathcal{S}(\mathbb{R}^n,\mathbb{R}_n)$, the Fourier transform of the gradient is
\begin{equation}\label{Eq_Fourier_P_1}
F[\nabla_au](\xi)=\mathbf{i}\,\xi_aF[u](\xi),\qquad\xi\in\mathbb{R}^n.
\end{equation}
Let us now consider the self-adjoint Laplace operator with constant coefficients
\begin{equation*}
-\Delta_a:=-\sum\limits_{i=1}^na_i^2\frac{\partial^2}{\partial x_i^2},\qquad\dom(-\Delta_a)=H^2(\mathbb{R}^n,\mathbb{C}_n).
\end{equation*}
With the Borel functional calculus, we can now define the inverse square root $(-\Delta_a)^{-\frac{1}{2}}$. This operator acts in Fourier space like the multiplication with $|\xi_a|^{-1}$. Since \eqref{Eq_Fourier_P_1} shows that $|\xi_a|^{-1}F[\nabla_au](\xi)\in L^2(\mathbb{R}^n,\mathbb{C}_n)$, there is $\nabla_au\in\dom(-\Delta_a)^{-\frac{1}{2}}$, and there is
\begin{equation}\label{Eq_Fourier_P_2}
F\big[(-\Delta_a)^{-\frac{1}{2}}\nabla_au\big](\xi)=\frac{1}{|\xi_a|}F[\nabla_au](\xi)=\frac{\mathbf{i}\,\xi_a}{|\xi_a|}F[u](\xi),\qquad\xi\in\mathbb{R}^n.
\end{equation}
On the other hand, by \cite[Theorem~3.23, Remark~3.3, Theorem~3.27]{CMS26}, there is
\begin{equation}\label{Eq_Fourier_P_3}
\sgn(\nabla_a)=\overline{q_{-1}(\nabla_a)p_1(\nabla_a)}=\overline{(-\Delta_a)^{-\frac{1}{2}}\nabla_a}.
\end{equation}
Note that the fractional power $(-\Delta_a)^{-\frac{1}{2}}$ in \eqref{Eq_Fourier_P_2} is understood as the Borel functional calculus of self-adjoint operators, while $(-\Delta_a)^{-\frac{1}{2}}$ in \eqref{Eq_Fourier_P_3} is understood as the $H^\infty$-functional calculus of sectorial operators. However, it is shown in \cite{S26} that these two functional calculi coincide. Combining now \eqref{Eq_Fourier_P_2} and \eqref{Eq_Fourier_P_3}, we get the following Fourier representation of the sign function
\begin{equation*}
F[\sgn(\nabla_a)u](\xi)=\frac{\mathbf{i}\,\xi_a}{|\xi_a|}F[u](\xi),\qquad u\in\mathcal{S}(\mathbb{R}^n,\mathbb{R}_n).
\end{equation*}
Since both sides are bounded operators which coincide on the Schwartz space, which is dense in $L^2(\mathbb{R}^n,\mathbb{R}_n)$, these operators have to coincide everywhere
\begin{equation}\label{Eq_Fourier_P_4}
F[\sgn(\nabla_a)u](\xi)=\frac{\mathbf{i}\,\xi_a}{|\xi_a|}F[u](\xi),\qquad u\in L^2(\mathbb{R}^n,\mathbb{R}_n).
\end{equation}
Finally, since we can write $\mathds{1}_\pm(s)=\frac{1\pm\sgn(s)}{2}$, the operator $P_\pm=\frac{1\pm\sgn(\nabla_a)}{2}$ admits the Fourier representation
\begin{equation*}
F[P_\pm u](\xi)=\frac{F[u](\xi)\pm F[\sgn(\nabla_a)u](\xi)}{2}=\frac{1}{2}\Big(1\pm\frac{\mathbf{i}\,\xi_a}{|\xi_a|}\Big)F[u](\xi). \qedhere
\end{equation*}
\end{proof}

As a consequence of this representation of the projectors, we can also represent the sign function of the gradient in Fourier space. This sign function fulfils an important task, namely it is the connection between the two different approaches of fractional powers of the gradient in \cite{CMS26}. These two approaches are connected to the functions
\begin{equation}\label{Eq_palpha_qalpha}
p_\alpha(s):=\begin{cases} s^\alpha, & \Sc(s)>0, \\ -(-s)^\alpha, & \Sc(s)<0, \end{cases}\qquad\text{and}\qquad q_\alpha(s):=\begin{cases} s^\alpha, & \Sc(s)>0, \\ (-s)^\alpha, & \Sc(s)<0. \end{cases}
\end{equation}
In the special case of the gradient with constant coefficients, the operator $\sgn(\nabla_a)$ is even the connection between the novel approach of fractional power $p_\alpha(\nabla_a)$ and the classical approach $(-\Delta_a)^{\frac{\alpha}{2}}$.

\begin{thm}\label{thm_sgn_connection}
For every $\alpha\in\mathbb{R}$, the fractional powers $p_\alpha(\nabla_a)$ and $q_\alpha(\nabla_a)$ of the gradient \eqref{Eq_Gradient} are connected via
\begin{equation}\label{Eq_palpha_gradient}
p_\alpha(\nabla_a)=q_\alpha(\nabla_a)\sgn(\nabla_a).
\end{equation}
In particular, if the coefficients $a_i(x)=a_i>0$ are constant, we get
\begin{equation}\label{Eq_palpha_Delta}
p_\alpha(\nabla_a)=(-\Delta_a)^{\frac{\alpha}{2}}\sgn(\nabla_a),\qquad\text{with}\qquad\sgn(\nabla_a)u=\mathbf{i}\,F^{-1}\Big[\frac{\xi_a}{|\xi_a|}F[u](\xi)\Big].
\end{equation}
\end{thm}

\begin{rem}
The operator $\sgn(\nabla_a)$ in \eqref{Eq_palpha_Delta} is known, in the Clifford setting, as weighted Hilbert transform, see \cite{Bernstein,QianYang2009Hilbert}. However, when applied to real valued functions, the operator $\sgn(\nabla_a)$ can be identified with a weighted Riesz transform if one identifies vectors in $\mathbb{R}^n$ with imaginary paravectors in $\mathbb{R}_n$.
\end{rem}

\begin{proof}[Proof of Theorem~\ref{thm_sgn_connection}]
In \cite[Theorem~3.23 \& Theorem~3.27]{CMS26} it is already proven that
\begin{equation}\label{Eq_palpha_gradient_1}
p_\alpha(\nabla_a)=\overline{q_\alpha(\nabla_a)p_0(\nabla_a)}.
\end{equation}
Since we know by Theorem~\ref{thm_Bounded_FC_gradient} that $\sgn(\nabla_a)$ is a bounded operator, the product with the closed operator $q_\alpha(\nabla_a)$ is closed again. This reduces \eqref{Eq_palpha_gradient_1} to the stated identity \eqref{Eq_palpha_gradient}. \medskip

In the case that $\nabla_a$ has constant coefficients, there is
\begin{equation}\label{Eq_qalpha_as_Laplace}
q_\alpha(\nabla_a)=(\nabla_a^2)^{\frac{\alpha}{2}}=(-\Delta_a)^{\frac{\alpha}{2}},
\end{equation}
where as stated in the discussion below \eqref{Eq_Fourier_P_3}, the $H^\infty$- and the Borel-functional calculus of $(-\Delta_a)^{\frac{\alpha}{2}}$ coincide. The Fourier representation of $\sgn(\nabla_a)$ is already shown in \eqref{Eq_Fourier_P_4}.
\end{proof}

The identity \eqref{Eq_palpha_Delta} now also allows us to identify the domains of the two versions $p_\alpha(\nabla_a)$ and $q_\alpha(\nabla_a)$ of fractional powers.

\begin{cor}
Let $\nabla_a$ be the gradient operator with constant coefficients $a_i(x)=a_i>0$. Then, for every $\alpha\in\mathbb{R}$, the domains of the fractional powers are given by
\begin{equation}\label{Eq_domain_equality}
\dom(p_\alpha(\nabla_a))=\dom(q_\alpha(\nabla_a))=\dom\big((-\Delta)^{\frac{\alpha}{2}}\big).
\end{equation}
In particular, the domain of $(-\Delta)^{\frac{\alpha}{2}}$ is in Fourier space given by
\begin{equation*}
\dom\big((-\Delta)^{\frac{\alpha}{2}}\big)=\Big\{u\in L^2(\mathbb{R}^n,\mathbb{R}_n)\;\Big|\;\int_{\mathbb{R}^n}|F[u](\xi)|^2|\xi|^{2\alpha}<\infty\Big\}.
\end{equation*}
In particular, for $\alpha\geq 0$, this domain coincides with the fractional Sobolev space $H^\alpha(\mathbb{R}^n)$.
\end{cor}

\begin{proof}
Using the connection $p_\alpha(s)=\sgn(s)q_\alpha(s)$ between the functions $p_\alpha$ and $q_\alpha$ in \eqref{Eq_palpha_qalpha}, as well as the connection between the domains in \cite[Theorem~5.7]{MS24}, we deduce
\begin{equation*}
\dom(\sgn(\nabla_a)q_\alpha(\nabla_a))=\dom(p_\alpha(\nabla_a))\cap\dom(q_\alpha(\nabla_a))\subseteq\dom(p_\alpha(\nabla_a)).
\end{equation*}
Moreover, since $\sgn(\nabla_a)$ is a bounded operator, we have
\begin{equation*}
\dom(\sgn(\nabla_a)q_\alpha(\nabla_a))=\dom(q_\alpha(\nabla_a)),
\end{equation*}
and combining these relations yields the inclusion
\begin{equation*}
\dom(q_\alpha(\nabla_a))\subseteq\dom(p_\alpha(\nabla_a)).
\end{equation*}
By symmetry, exchanging the roles of $p_\alpha$ and $q_\alpha$ gives the reverse inclusion, and therefore the domains coincide. For the second domain equality in \eqref{Eq_domain_equality}, we already know from \eqref{Eq_qalpha_as_Laplace} that $q_\alpha(\nabla_a)=(-\Delta_a)^{\frac{\alpha}{2}}$ acts as the Laplacian with constant coefficients. Since $-\Delta_a$ differs from the Laplacian $-\Delta $ only by finitely many positive constant coefficients $a_1,\dots,a_n>0$, the domains of their fractional powers are the same. This then leads to
\begin{equation*}
\dom(q_\alpha(\nabla_a))=\dom((-\Delta_a)^{\frac{\alpha}{2}})=\dom((-\Delta)^{\frac{\alpha}{2}}).
\end{equation*}
Finally, the explicit characterization of the domain of $(-\Delta)^{\frac{\alpha}{2}}$ follows by applying the Fourier transform.
\end{proof}

After we found a representation for the projectors $P_\pm$ of the gradient in Theorem~\ref{thm_Fourier_P}, we will now also characterize the range $V_\pm$ of these projectors.

\begin{thm}
Let $\nabla_a$ be the gradient with constant coefficients $a_i(x)=a_i>0$. Then, the subspaces $V_\pm$ from Definition~\ref{defi_V} admit the explicit representation
\begin{equation*}
V_\pm=\Big\{F^{-1}\Big[\Big(1\pm\frac{\mathbf{i}\,\xi_a}{|\xi_a|}\Big)v(\xi)\Big]\in L^2(\mathbb{R}^n,\mathbb{R}_n)\;\Big|\;v\in L^2(\mathbb{R}^n,\mathbb{C}_n)\Big\}.
\end{equation*}
Note that we only choose those $\mathbb{C}_n$-valued functions $v$, for which the inverse Fourier transform is $\mathbb{R}_n$-valued.
\end{thm}

\begin{proof}
Since $P_\pm^2=P_\pm$ by Lemma~\ref{lem_P_properties}~a), the range of the projectors is given by
\begin{equation*}
V_\pm=\ran P_\pm=\big\{u\in L^2(\mathbb{R}^n,\mathbb{R}_n)\;\big|\;P_\pm u=u\big\}.
\end{equation*}
Next, by the Fourier representation of the projectors in Theorem~\ref{thm_Fourier_P}, we can rewrite
\begin{equation}\label{Eq_V_gradient_2}
V_\pm=\Big\{u\in L^2(\mathbb{R}^n,\mathbb{R}_n)\;\Big|\;\Big(1\mp\frac{\mathbf{i}\,\xi_a}{|\xi_a|}\Big)F[u](\xi)=0\Big\}.
\end{equation}
It follows immediately from the two basic identities
\begin{equation*}
\Big(1\mp\frac{\mathbf{i}\,\xi_a}{|\xi_a|}\Big)\Big(1\pm\frac{\mathbf{i}\,\xi_a}{|\xi_a|}\Big)=0,\qquad\text{and}\qquad\frac{1}{2}\bigg(\Big(1+\frac{\mathbf{i}\,\xi_a}{|\xi_a|}\Big)+\Big(1-\frac{\mathbf{i}\,\xi_a}{|\xi_a|}\Big)\bigg)=1,
\end{equation*}
that
\begin{equation*}
\Big(1\mp\frac{\mathbf{i}\,\xi_a}{|\xi_a|}\Big)F[u](\xi)=0\qquad\Leftrightarrow\qquad F[u](\xi)=\Big(1\pm\frac{\mathbf{i}\,\xi_a}{|\xi_a|}\Big)v(\xi),\quad\text{for some }v\in L^2(\mathbb{R}^n,\mathbb{C}_n).
\end{equation*}
Consequently, the kernel representation of $V_\pm$ in \eqref{Eq_V_gradient_2} can be rewritten as
\begin{align*}
V_\pm&=\Big\{u\in L^2(\mathbb{R}^n,\mathbb{R}_n)\;\Big|\;\exists v\in L^2(\mathbb{R}^n,\mathbb{C}_n):\;F[u](\xi)=\Big(1\pm\frac{\mathbf{i}\,\xi_a}{|\xi_a|}\Big)v(\xi)\Big\} \\
&=\Big\{F^{-1}\Big[\Big(1\pm\frac{\mathbf{i}\,\xi_a}{|\xi_a|}\Big)v(\xi)\Big]\in L^2(\mathbb{R}^n,\mathbb{R}_n)\;\Big|\;v\in L^2(\mathbb{R}^n,\mathbb{C}_n)\Big\}. \qedhere
\end{align*}
\end{proof}

\hspace{1cm}

\textbf{Acknowledgements.}
Francesco Mantovani is supported by MUR grant Dipartimento di Eccellenza 2023-2027.

\end{document}